\documentclass[document]{amsart}
\usepackage{amsmath,amssymb,amsthm,enumerate,graphicx}
\usepackage[mathscr]{eucal}
\usepackage{tabularx}
\usepackage[margin=1.3in]{geometry}
\usepackage{tikz}
\usepackage{caption}
\usepackage{color}
\usepackage{subcaption}
\usepackage{pgfplots}
\pgfplotsset{compat=1.6}
\pgfplotsset{soldot/.style={color=black,only marks,mark=*}} \pgfplotsset{holdot/.style={color=black,fill=white,only marks,mark=*}}

\newtheorem{theorem}{Theorem}[section]
\newtheorem{lemma}[theorem]{Lemma}
\newtheorem{prop}[theorem]{Proposition}

\theoremstyle{definition}
\newtheorem{definition}[theorem]{Definition}
\newtheorem{example}[theorem]{Example}

\theoremstyle{remark}

\numberwithin{equation}{section}

\newcommand{\Q}{\mathbb{Q}}
\newcommand{\R}{\mathbb{R}}

\newcommand{\C}{\mathbb{C}}
\newcommand{\N}{\mathbb{N}}

\newcommand{\T}{\mathbb{T}}
\newcommand{\E}{\mathbb{E}}

\newcommand{\lra}{\longrightarrow}

\newcommand{\iv}{^{-1}}
\newcommand{\bs}{\backslash}

\newcommand{\card}{\text{card}}
\newcommand{\supp}{\text{Supp}}
\newcommand{\diag}{\text{Diag}}
\newcommand{\dist}{\text{dist}}

\begin{document}

\title{The spectrum of large unitarily invariant models with increasingly many spikes}
\author{Brady Thompson}
\date{}
\begin{abstract}
In this paper we study random matrix models where the matrices in question contain infinitely many spikes. Recent work has characterized the possible outliers in the spectrum of large deformed unitarily invariant models when the number of spikes in the model is fixed. We show that similar results hold when the number of spikes grows along with the size of the matrix and these spikes accumulate to the support of the limiting eigenvalue distribution. 
\end{abstract}
\maketitle


\section{Introduction}
	
Given the spectrum of two $N \times N$ Hermitian matrices $A_N$ and $B_N$, discovering the spectrum of $A_N + B_N$ is a rather difficult procedure. If we add a bit of randomness to the model, then free probability provides a helpful description of the spectrum of the sum.  The connection between free probability and random matrices was first made in the seminal work of Voiculescu \cite{voiculescu1986addition,voiculescu1987multiplication,voiculescu1991limit}. The tools developed in free probability theory provide a natural framework for studying the eigenvalue distribution of random matrix models. 

One of the models we consider is $X_N = A_N + U_N B_N U_N^*$ where $U_N$ is a $N \times N$ random unitary matrix distributed according to the Haar measure on the unitary group U$(N)$, often called a Haar unitary matrix. Furthermore, we suppose that the empirical eigenvalue distribution
\[\mu_{A_N} := \frac{1}{N}\sum_{i=1}^{N}\delta_{\lambda_i({A_N})}\] of $A_N$ and $\mu_{B_N}$ of $B_N$ converge weakly to the compactly supported measures $\mu$ and $\nu$ respectively. Speicher proved in \cite{speicher1993free} that $\mu_{X_N}$ converges weakly to the free additive convolution $\mu\boxplus\nu$ as $N \to \infty$ \cite{nica2006lectures,mingo2017free,voiculescu1992free}. Even though the empirical eigenvalue distribution of $X_N$ converges weakly to $\mu \boxplus \nu$, this does not necessarily imply that all the eigenvalues of $X_N$ converge to the support of $\mu \boxplus \nu$. Building on a series of results  about strong convergence of random matrices \cite{haagerup2005new,male2012norm,collins2014strong}, Collins and Male \cite{collins2014strong} provided conditions under which the eigenvalues of $X_N$ uniformly converge to the support of $\mu \boxplus \nu$. Their result states that, for independent Hermitian random matrices $A_N$ and $U_NB_NU_N^*$, if almost surely, the eigenvalues of $A_N$  uniformly converge to $\supp(\mu)$, and, almost surely, the eigenvalues of $U_NB_NU_N^*$ uniformly converge to $\supp(\nu)$, then the eigenvalues of $A_N + U_NB_NU_N^*$ uniformly converge to $\supp(\mu \boxplus \nu)$ almost surely. 

Our central problem is to examine particular situations where strong convergence does not occur. One such situation is realized by considering a finite sequence, $\theta_1, \dots, \theta_p$, of real numbers, such that $\theta_1, \dots, \theta_p \in \sigma(A_N)$ for every $N$, and such that each $\theta_i$ lies outside the support of $\mu$. We also impose that except for these values, the remaining eigenvalues converge uniformly to $\supp(\mu)$. These values, $\theta_1, \dots, \theta_p,$ are called the \textit{spikes}. 

Belinschi, Bercovici, Capitaine, and Fevrier gave a description of the limiting spectrum $\sigma(X_N)$ when both matrices $A_N$ and $B_N$ have a finite number of spikes \cite{belinschi2017outliers}. The theory of subordination plays a crucial role.  The defining property of the subordination functions, $\omega_1$ and $\omega_2$, is their relationship with the Cauchy transform. Namely, given two compactly supported probability measures, $\mu$ and $\nu$, we have
\[G_{\mu \boxplus \nu}(z) = G_{\mu}(\omega_1(z))  = G_{\nu}(\omega_2(z)).\]
For further information on subordination, we refer the reader to chapter 2 of \cite{mingo2017free} and section 3.4 of \cite{belinschi2017outliers}. The result of \cite{belinschi2017outliers} states that, almost surely, for any neighborhood $E$ of the set
\[ \displaystyle \supp(\mu \boxplus \nu)\cup \left[ \bigcup_{i=1}^p \omega_1\iv(\{\theta_i\})\right] \cup \left[ \bigcup_{j=1}^q \omega_2\iv(\{\tau_i\})\right] \]
there exists an $N_0 \in \N$ such that for all $N \geq N_0$ we have 
\[\sigma(X_N) \subset E.\]

We generalize the results of \cite{belinschi2017outliers}  to the case when there is an increasing number of spikes that grows along with the size of the matrix. We limit ourselves to the case when the spikes accumulate to boundary of the support of the limiting eigenvalue distribution, that is, 
\begin{enumerate}[\indent {$-$}]
\item $\theta_i$ does not belong to $\supp(\mu)$ for all $i = 1, 2, \dots$ and,
\item  $\dist(\theta_i, \supp(\mu)) \lra 0$ as $i \lra \infty$. 
\end{enumerate}
Our result states that precisely the same conclusions from \cite{belinschi2017outliers} hold regarding the asymptotic behavior of the eigenvalue distribution for the following models:
\begin{enumerate}[\indent {$-$}]
\item $X_N = A_N + U_NB_NU_N^*$ where $A_N = A_N^*$, $B_N = B_N^*$ are random matrices and $U_N$ is a Haar-distributed unitary random matrix;
\item $X_N = A_N^{1/2}U_NB_NU_N^*A_N^{1/2}$ where $A_N$, $B_N \geq 0$ are random matries, and $U_N$ is a Haar-distributed unitary random matrix;
\item $X_N = A_NU_NB_NU_N*$  where $A_N, B_N,U_N \in $ U$(N)$ are Haar-distributed unitary random matrices.
\end{enumerate}


\section{Statement of Main Results}
	The result of Collins and Male (Corollary 2.2 of \cite{collins2014strong}) mentioned in the introduction is an indispensable tool for the proof of our results. It is used several times throughout the paper and is given as a theorem below. 
\begin{theorem}\label{cm}
Let $A_N$ and $B_N$ be independent Hermitian random matrices. Assume that:
\begin{enumerate}[$1$.]
\item the law of one of the matrices is invariant under unitary conjugacy;
\item almost surely, the empirical eigenvalue distribution $\mu_{A_N}$ (respectively $\mu_{B_N}$) converges to a compactly supported probability measure $\mu$ (respectively $\nu$);
\item almost surely, for any neighborhood $E$ of the support of $\mu$ (respectively $\nu$), for large enough $N$, the eigenvalues of $A_N$ (respectively $B_N$) belong to $E$. 
\end{enumerate}
Then;
\begin{enumerate}[\indent {$-$}]
\item almost surely, for any neighborhood $F$ of the support of $\mu \boxplus \nu$, for $N$ large enough, the eigenvalues of $A_N+B_N$ belong to $F$;
\item if moreover $A_N$ is nonnegative, almost surely, for any neighborhood $G$ of the support of $\mu \boxtimes \nu$, for $N$ large enough, the eigenvalues of $A_N^{1/2}B_N A_N^{1/2}$ belong to $G$. 
\end{enumerate}
\end{theorem}

We are interested in a case where condition (3) of the preceding theorem is not satisfied. More precisely, with respect to the additive matrix model, $X_N = A_N +U_NB_NU_N^*$, our goal is to describe the eigenvalue distribution of $X_N$ where the matrix $A_N$ has  spikes $\theta_1, \theta_2, \dots$, and $B_N$ also has spikes $\tau_1, \tau_2, \dots$. The case of finitely many spikes was considered by \cite{belinschi2017outliers}, and their result is given as a theorem below.  We use the notation that for any set $E \subset \R$, we have $E_\varepsilon := E+ (-\varepsilon, \varepsilon)$.

\begin{theorem}\label{bbcf} Suppose the following,
\begin{enumerate}[$1$.]
\item Two compactly supported Borel probability measures $\mu$ and $\nu$ on $\R$. 
\item A positive integer $p$ and fixed real numbers 
\[\theta_1 \geq \theta_2 \geq \dots \geq \theta_p\]
which do not belong to $\text{\emph{Supp}}(\mu)$. 
\item A sequence $\{A_N\}_{N\in\N}$ of deterministic Hermitian matrices of size $N\times N$ such that:
\begin{enumerate}[\indent {$-$}]
\item $\mu_{A_N}$ converges weakly to $\mu$ as $N \lra \infty$;
\item for $N\geq p$ and $\theta \in \{ \theta_1, \dots, \theta_p\}$, the sequence $(\lambda_n(A_N))_{n=1}^{N}$ satisfies 
\[\text{\emph{card}}(\{n : \lambda_N(A_N) = \theta\}) = \text{\emph{card}}(\{i : \theta_i = \theta\});\]
\item the eigenvalues of $A_N$ which are not equal to some $\theta_i$ converge uniformly to $\text{\emph{Supp}}(\mu)$ as $N \to \infty$, that is 
\[\lim_{N \to \infty} \max_{\lambda_n(A_N) \notin \{\theta_1,\dots, \theta_{p}\}} \text{\emph{dist}}(\lambda_n(A_N),\text{\emph{Supp}}(\mu)) = 0.\]
\end{enumerate}
\item A positive integer $q$ and fixed real numbers
\[\tau_1 \geq \tau_2 \geq \dots \geq \tau_q\]
which do not belong to $\text{\emph{Supp}}(\nu)$.
\item A sequence $\{B_N\}_{N \in \N}$ of deterministic Hermitian matrices of size $N \times N$ such that
\begin{enumerate}[\indent {$-$}]
\item $\mu_{B_N}$ converges weakly to $\nu$ as $N \lra \infty$;
\item for $N\geq q$ and $\tau \in \{ \tau_1, \dots, \tau_q\}$, the sequence $(\lambda_n(B_N))_{n=1}^{N}$ satisfies 
\[\text{\emph{card}}(\{n : \lambda_N(B_N) = \tau\}) = \text{\emph{card}}(\{j : \tau_j = \tau\});\]
\item the eigenvalues of $B_N$ which are not equal to some $\tau_j$ converge uniformly to $\text{\emph{Supp}}(\nu)$ as $N \to \infty$.
\end{enumerate}
\item A sequence $\{U_N\}_{N \in \N}$ of unitary random matrices such that the distribution of $U_N$ is the normalized Haar measure on the unitary group U$(N)$.
\end{enumerate}
With the above notations, set $K =\text{\emph{Supp}}(\mu \boxplus \nu)$ and 
\[K' = \displaystyle K \cup \left[ \bigcup_{i=1}^p \omega_1\iv(\{\theta_j\})\right] \cup \left[ \bigcup_{j=1}^q \omega_2\iv(\{\tau_i\})\right] \]
where $\omega_1$ and $\omega_2$ are the subordination functions. Define $X_N = A_N + U_NB_NU_N^*$. 

Then,

\begin{enumerate}[$1$.]
\item Given $\varepsilon >0$, almost surely, there exists an $N_0 \in \N$ such that for all $N >N_0$, we have \[\sigma(X_N) \subset K'_{\varepsilon}.\]
\item  Fix a number $\rho \in K'\bs K$. Let $\varepsilon > 0 $ such that $(\rho - 2\varepsilon , \rho+2\varepsilon)\cap K' = \{\rho\}$ and set $k = \text{\emph{card}}(\{i:\omega_1(\rho) = \theta_i\})$ and $\ell= \text{\emph{card}}(\{j:\omega_2(\rho) = \tau_j\})$ then almost surely, there exists an $N_0 \in \N$ such that for all $N>N_0$, we have \[\text{\emph{card}}(\{\sigma(X_N) \cap (\rho-\varepsilon,\rho +\varepsilon)\}) = k+\ell.\]

\end{enumerate}
\end{theorem} 

These elements of $K'\bs K$ are called the \textit{outliers} of the model. This result demonstrates that we can use the subordination functions to calculate the outliers that arise from the spiked model.We give the following example. 

\begin{example}\label{example1}

For this numerical simulation we let $N = 1000$. Let $A_N = (3I_{N/2})\oplus(-3I_{N/2})$. Define the matrix $B_N$ as 
\[B_N =  \begin{bmatrix}
     \frac{W}{\sqrt{N-2}} & 0_{(N-2)\times 1} & 0_{(N-2)\times 1} \\
    0_{1\times (N-2)}  & -10 & 0\\
    0_{1\times (N-2)} & 0 & 7
\end{bmatrix}\]
where $W$ is a $998\times998$ GUE.  The histogram of the eigenvalues of one sample $X_N = B_N + U_NB_NU_N^*$ is shown in  Figure \ref{figure1}. We see the presence of four outliers in the distribution, and using Theorem \ref{bbcf} we can calculate them explicitly. 

\begin{figure}[h]
\includegraphics[width=10cm]{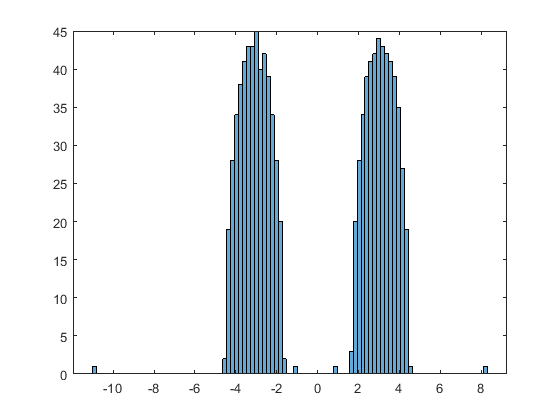}
\caption{Eigenvalue Distribution of $X_N$ where $B_N$ has spikes}\label{figure1}
\end{figure}

We know that $\mu_{A_N} \lra \mu = \frac{1}{2}(\delta_{-3} + \delta_{3}$) and $\mu_{B_N}$ converges weakly to the semicircle distribution, which we will denote as $s$. Computing the Cauchy transform of $\mu$ gives,
\[G_{\mu}(z) = \int_{\R} \frac{1}{z-t}d\mu(t) = \frac{1}{2}\left( \frac{1}{z+3} +\frac{1}{z-3}\right) = \frac{z}{z^2 - 9}. \]
Using the fact that the $R$-transform of the semicircle distribution gives $R_{s}(z) = z$, and the compositional inverse of the Cauchy transform is $G_{\nu}^{\langle -1 \rangle}(z) = R_{\nu}(z) + \frac{1}{z}$, we get
\[G_{s}(z) = \frac{z\pm\sqrt{z^2-4}}{2}.\]
We choose the plus or minus depending of the sign of $z$. Recall the property of the subordination functions that 
\begin{align}
\omega_{\mu}(z) + \omega_{s}(z) - z &= \frac{1}{G_{\mu}(\omega_{\mu}(z))},\label{1.1}\\
\omega_{\mu}(z) + \omega_{s}(z) - z &= \frac{1}{G_{s}(\omega_{s}(z))}.\label{1.2} 
\end{align}
If we first consider the case where $\omega_s(z) > 0$. Solving equation (\ref{1.1}) for $\omega_{\mu}(z)$ gives
\[\omega_{\mu}(z) = \frac{-9}{\omega_s(z) - z}\]
and substituting into (\ref{1.2}) and solving for $z$ gives the following equation
\begin{equation}
z = \omega_s(z) - \frac{1 \pm \sqrt{1+9\cdot4\left(\frac{\omega_s(z) - \sqrt{\omega_s(z)^2 - 4}}{2}\right)^2}}{\omega_s(z) -\sqrt{\omega_s(z)^2 - 4}}.
\end{equation}
Similarly, the case where $\omega_s(z) < 0 $ we get 
\begin{equation}\label{CalculateOutliers}
z = \omega_s(z) - \frac{1 \pm \sqrt{1+9\cdot4\left(\frac{\omega_s(z)+ \sqrt{\omega_s(z)^2 - 4}}{2}\right)^2}}{\omega_s(z) + \sqrt{\omega_s(z)^2 - 4}}.
\end{equation}

Theorem \ref{bbcf} indicates that the relationship between spikes $\tau$ of $B_N$,  and the outliers $\rho$ is $\omega_s(\rho) = \tau$. Hence, if we evaluate equation (\ref{CalculateOutliers}) at $z = \rho$, then we see that the spike $\tau_1 = 7$ produces the outliers 
\[\rho_1 = 7- \frac{1 + \sqrt{1+9\cdot4\left(\frac{7- \sqrt{7^2 - 4}}{2}\right)^2}}{7-\sqrt{7^2 - 4}} \approx -0.9817\]
and 
\[\rho_2 = 7- \frac{1 - \sqrt{1+9\cdot4\left(\frac{7 - \sqrt{7^2 - 4}}{2}\right)^2}}{7-\sqrt{7^2 - 4}} \approx 8.1276.\]

Similarly, calculating the outliers that are generated by the spike $\tau_2 = -10$ gives $\rho_3 \approx 0.7372$ and $\rho_4 \approx  -10.8382$. 

We can also suppose that $A_N$ is spiked, for example, 
\[A_N =  \begin{bmatrix}
     3I_{499}		& 0_{499\times 499} 		& 0_{499 \times 1} 		&  0_{499 \times 1} \\
    0_{499\times 499}  & -3I_{499} 		& 0 _{499 \times 1}		 & 0_{499 \times 1} \\
    0_{1\times 499}  	 & 0_{1\times 499}	 & -5				&0 \\
    0_{1\times 499}  & 0_{1\times 499}   	& 0 					&6\\
\end{bmatrix}\]
We notice in Figure \ref{figure2} the presence of two additional spikes. 

\begin{figure}[h]
\includegraphics[width=10cm]{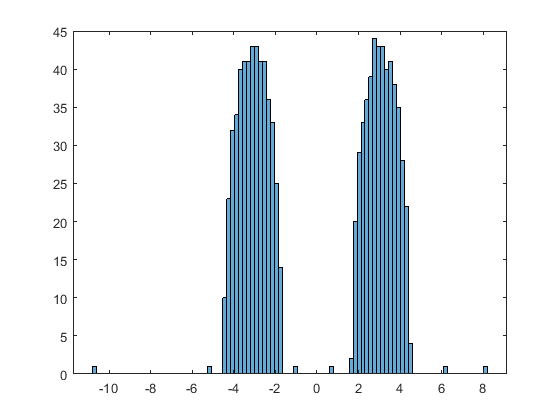}
\caption{Eigenvalue Distribution of $X_N$ where both $A_N$ and $B_N$ have spikes}\label{figure2}
\end{figure}

In order to calculate the outliers produced by these spikes we follow the same procedure as above. Solving for $\omega_s(z)$ in (\ref{1.2}) gives 
\[\omega_s(z) = \frac{-1 -\omega_{\mu}(z)^2 + 2\omega_{\mu}(z) z - z^2}{\omega_{\mu}(z) - z}\]
and substituting into equation (\ref{1.1}) and solving for $z$ gives
\[\frac{-1 -\omega_{\mu}(z)^2 + 2\omega_{\mu}(z) z - z^2}{\omega_{\mu}(z) - z} + \omega_{\mu}(z) -z = \frac{\omega_{\mu}(z)^2-9}{\omega_{\mu}(z)}\]
and 
\[z = \frac{\omega_{\mu}(z)  ( \omega_{\mu}(z) ^2-8)}{ \omega_{\mu}(z) ^2-9}.\]
Theorem \ref{bbcf} indicates that the relationship between spikes $\theta$ of $A_N$ and outliers $\rho$ is $\omega_{\mu}(\rho) = \theta$. When we evaluate the above equation at $z = \rho$, we see that the spike $\theta_1 = -5$ produces the outlier
\[\rho_5 = \frac{-5  ( (-5)^2-8)}{ (-5)^2-9} = -5.3125\]
and the spike $\theta_2 = 6$ produces the outlier $\rho_6 = \frac{56}{9} \approx 6.2222.$

\end{example}

It is our goal to extend this result to the case when the number of spikes in  $A_N$ and $B_N$ tends to infinity as $N \to \infty$. More precisely, let $(\theta_i)_{i \in \N}$ be a real-valued sequence such that $\theta_i \notin \supp(\mu)$ for $i = 1, 2, \dots$ and $\dist(\theta_i, \supp(\mu)) \to 0$ as $i \to \infty$. Also, let $(\tau_j)_{j\in \N}$ be a real-valued sequence such that $\tau_j \notin \supp(\nu)$ for $j = 1,2, \dots$ and $\dist(\tau_j, \supp(\nu)) \to 0$ as $j \to \infty$. To ensure that these spikes do not influence the limiting empirical spectral distributions of $A_N$ and $B_N$, and hence that of $X_N$, we have to control the rate at which we add additional spikes to the model. 

\begin{prop}\label{Qseq}
Let $(\varphi(N))_{N \in \N}$ be a monotonically increasing, nonnegative integer-valued sequence such that
\begin{enumerate}[$1$.]
\item $\displaystyle \frac{\varphi(N)}{N} \to 0$ as $N \to \infty$;
\item $N \geq \varphi(N)$ for every $N \in \N$. 
\end{enumerate}
Let $D_N$ be a deterministic diagonal matrix, $D_N = \displaystyle \text{\emph{Diag}}\left(\alpha_1^{(N)}, \dots, \alpha_{N}^{(N)}\right)$ such that $\mu_{D_N}$ converges weakly to a compactly supported measure $\mu$ as $N \lra \infty$. Then the spiked sequence of matrices 
\[\widetilde{D}_N =  \displaystyle \text{\emph{Diag}}\left(\theta_1, \dots, \theta_{\varphi(N)}, \alpha_1^{(N)}, \dots, \alpha_{\varphi(N) - N}^{(N)}\right) \]
also has the property that $\mu_{\widetilde{D}_N}$ converges weakly to $\mu$ as $N \to \infty$.
\end{prop}
\begin{proof}

Notice that 
\[\mu_{\tilde{D}_N} = \sum_{i=1}^{\varphi(N)} \frac{1}{N} \delta_{\theta_i} + \sum_{i=1}^{N-\varphi(N)} \frac{1}{N} \delta_{\alpha_i^{(N)}}.\]
Computing the total variation of $\mu_{D_N} - \mu_{\widetilde{D}_N}$ gives
\begin{align*}
\big|\big| \mu_{D_N} - \mu_{\widetilde{D}_N} \big|\big| &= \huge|\huge|\sum_{i=1}^{N} \frac{1}{N} \delta_{\alpha_i^{(N)}} -\sum_{i=1}^{\varphi(N)} \frac{1}{N} \delta_{\theta_i} - \sum_{i=1}^{N-\varphi(N)} \frac{1}{N} \delta_{\alpha_i^{(N)}}\huge|\huge|\\
&= \huge|\huge|    \sum_{i=N-\varphi(N)+1}^{N} \frac{1}{N} \delta_{\alpha_i^{(N)}} -\sum_{i=1}^{\varphi(N)} \frac{1}{N} \delta_{\theta_i}   \huge|\huge| \\
& \leq \huge|\huge|      \sum_{i=N-\varphi(N)+1}^{N} \frac{1}{N} \delta_{\alpha_i^{(N)}}         \huge|\huge| + \huge|\huge|          \sum_{i=1}^{\varphi(N)} \frac{1}{N} \delta_{\theta_i}           \huge|\huge| \\
& = \frac{\varphi(N)-1}{N} + \frac{\varphi(N)}{N} \to 0 ~~  \text{ as }~~ N \to \infty\\
\end{align*}
and since $\mu_{D_N}$ converges weakly to $\mu$, so does $\mu_{\widetilde{D}_N}$. 

\end{proof}

With this proposition in hand, we can now state the main result for the additive model. The results for the multiplicative cases are given in their respective section. 
\begin{theorem}\label{main}
Suppose we have the following:
\begin{enumerate}[$1$.]
\item Two compactly supported Borel probability measures $\mu$ and $\nu$ on $\R$. 

\item A sequence of fixed real numbers $\{\theta_i\}_{i=1}^{\infty}$ such that:
\begin{enumerate}[\indent {$-$}]
\item $\theta_i$ does not belong to $\text{\emph{\supp}}(\mu)$ for all $i = 1, 2, \dots$;
\item  \text{\emph{dist}}$(\theta_i, \text{\emph{Supp}}(\mu)) \lra 0$ as $i \lra \infty$. 
\end{enumerate}

\item A sequence $\{A_N\}_{N\in\N}$ of random Hermitian matrices of size $N\times N$ such that:
\begin{enumerate}[\indent {$-$}]
\item $\mu_{A_N}$ converges weakly to $\mu$ as $N \lra \infty$;
\item a sequence $(\varphi(N))_{N \in \N}$ satisfying the conditions in  Proposition \ref{Qseq};
\item for $\theta \in \{\theta_1, \dots, \theta_{\varphi(N)}\}$, the sequence $(\lambda_n(A_N))_{n=1}^{N}$ satisfies 
\[\text{\emph{card}}(\{n : \lambda_n(A_N) = \theta\}) = \text{\emph{card}}(\{i : \theta_i = \theta\});\]
\item the eigenvalues of $A_N$ which are not equal to some $\theta_i$ converge uniformly to $\text{\emph{Supp}}(\mu)$ as $N \lra \infty$, that is 
\[\lim_{N \lra \infty} \max_{\lambda_n(A_N) \notin \{\theta_1,\dots, \theta_{\varphi(N)}\}} \text{\emph{dist}}(\lambda_n(A_N),\text{\emph{Supp}}(\mu)) = 0.\]
\end{enumerate}

\item A sequence $\{U_N\}_{N \in \N}$ of unitary random matrices such that the distribution of $U_N$ is the normalized Haar measure on the unitary group U$(N)$.

\item A sequence of fixed real numbers $\{\tau_j\}_{j=1}^{\infty}$ such that:
\begin{enumerate}[\indent {$-$}]
\item $\tau_j$ does not belong to $\text{\emph{Supp}}()\nu)$ for all $j = 1, 2, \dots$;
\item  \text{\emph{dist}}$(\tau_j, \text{\emph{Supp}}(\nu)) \lra 0$ as $j \lra \infty$. 
\end{enumerate}

\item A sequence $\{B_N\}_{N\in\N}$ of random Hermitian matrices of size $N\times N$ such that:
\begin{enumerate}[\indent {$-$}]
\item $\mu_{B_N}$ converges weakly to $\nu$ as $N \lra \infty$;
\item a sequence $(\psi(N))_{N \in \N}$ satisfying the conditions in  Proposition \ref{Qseq};
\item for $\tau \in \{\tau_1, \dots, \tau_{\psi(N)}\}$, the sequence $(\lambda_n(B_N))_{n=1}^{N}$ satisfies 
\[\text{\emph{card}}(\{n : \lambda_n(B_N) = \tau\}) = \text{\emph{card}}(\{j : \tau_j = \tau\});\]
\item the eigenvalues of $B_N$ which are not equal to some $\tau_j$ converge uniformly to $\text{\emph{Supp}}(\nu)$ as $\N \lra \infty$, that is 
\[\lim_{N \lra \infty} \max_{\lambda_n(B_N) \notin \{\theta_1,\dots, \theta_{\psi(N)}\}} \text{\emph{dist}}(\lambda_n(B_N),\text{\emph{Supp}}(\nu)) = 0.\]
\end{enumerate}
\end{enumerate}

Set $K = \text{\emph{Supp}}(\mu \boxplus \nu)$ and $\displaystyle K' = K \cup \left[ \bigcup_{i=1}^{\infty} \omega_1^{-1} (\{\theta_i\})\right] \cup \left[ \bigcup_{j=1}^{\infty} \omega_2^{-1} (\{\tau_j\})\right]$ where $\omega_1,\omega_2$ are the subordination functions. Then,
\begin{enumerate}[$1$.]
\item Given $\varepsilon >0$, almost surely, there exists an $N_0 \in \N$ such that for all $N >N_0$, we have \[\sigma(X_N) \subset K'_{\varepsilon}.\]
\item  Fix a number $\rho \in K'\bs K$. Let $\varepsilon > 0 $ such that $(\rho - 2\varepsilon , \rho+2\varepsilon)\cap K' = \{\rho\}$ and set $k = \text{\emph{card}}(\{i:\omega_1(\rho) = \theta_i\})$ and $\ell = \text{\emph{card}}(\{j:\omega_2(\rho) = \tau_j\})$, then almost surely, there exists an $N_0 \in \N$ such that for all $N>N_0$, we have \[\text{\emph{card}}(\{\sigma(X_N)\cap (\rho-\varepsilon,\rho +\varepsilon)\}) = k+l.\]
\end{enumerate}
\end{theorem}

In a similar result, using the concept of cyclic monotone independence,  Collins, Hasebe, and Sakuma address the special case when $A_N$ and $B_N$ are made entirely of spikes with the property that the spikes are selected from a sequence that converges to zero.


\section{Preliminary Results}
	
Before proving Theorem \ref{main}, we establish a preliminary result regarding the convergence of the empirical spectral distribution of a sum of random Hermitian matrices. 

\begin{theorem}\label{conv}
Let $A_N$ and $B_N$ be $N \times N$ independent Hermitian random matrices such that the laws of $A_N$ and $B_N$ are unitarily invariant. Let $\mu$ and $\nu$ be compactly supported measures on $\R$ such that almost surely $\mu_{A_N} \lra \mu$ and almost surely $\mu_{B_N} \lra \nu$. Let $\varepsilon > 0$ be given. Suppose there exists an $N_0 \in \N$ large enough such that,  for $N \geq N_0$, we have both 
\[\sigma(A_N) \subseteq \text{\emph{Supp}}(\mu)_{\varepsilon}\]
\[ \sigma(B_N) \subseteq \text{\emph{Supp}}(\nu)_{\varepsilon}.\]
Then, for any $\eta > 0$, there exists an $N_1\in \N$ such that for all $N \geq N_1$, we have $\sigma(A_N+B_N) \subseteq \text{Supp}(\mu\boxplus\nu)_{2\varepsilon + \eta}$.
\end{theorem}

\begin{proof}
Let $\varepsilon > 0$ be such that there exists an $N_0$ such that for all $N \geq N_0$, we have $\sigma(A_N) \subseteq \text{Supp}(\mu)_{\varepsilon}$, and $\sigma(B_N) \subseteq \text{Supp}(\nu)_{\varepsilon}$. For such $N$'s, let $\theta_1^{(N)}, \dots, \theta_{k(N)}^{(N)}$ be all the points that are in $\sigma(A_N)\backslash \text{Supp}(\mu)$, and for $B_N$, let $\tau_1^{(N)}, \dots, \tau_{l(N)}^{(N)} $ be all the points that are in $\sigma(B_N)\backslash \text{Supp}(\nu)$. Since we have that the distributions are unitarily invariant, and we can reduce to the case where the matrices $A_N$ and $B_N$ have the form 
\[A_N = \displaystyle \text{diag}\left(\theta_1^{(N)}, \dots, \theta_{k(N)}^{(N)}, \alpha_1^{(N)}, \dots, \alpha_{N-k(N)}^{(N)}\right)\]
\[B_N = \displaystyle \text{diag}\left(\tau_1^{(N)}, \dots, \tau_{l(N)}^{(N)}, \beta_1^{(N)}, \dots, \beta_{N-l(N)}^{(N)}\right).\]

For the remainder of the proof we drop the superscript $(N)$ for convenience. Consider the numbers 
$a_j$ to be an element in $\supp(\mu)$ such that the difference $|\theta_j - a_j|$ is a minimum. Similarly, consider the numbers $b_j$ to be an element in $\supp(\nu)$ such that the difference $|\tau_j - b_j|$ is a minimum. Define the matrices 

\[A_N' := \displaystyle \text{diag}\left(a_1, \dots, a_{k(N)}, \alpha_1, \dots, \alpha_{N-k(N)}\right)\]
\[B_N ':= \displaystyle \text{diag}\left( b_1, \dots, b_{l(N)}, \beta_1, \dots, \beta_{N-l(N)}\right)\]
and 

\[A_N'' := \displaystyle \text{diag}\left(\theta_1-a_1, \dots, \theta_{k(N)}-a_{k(N)}, 0, \dots, 0 \right)\]
\[B_N '':= \displaystyle \text{diag}\left(\tau_1 - b_1, \dots, \tau_{l(N)} - b_{l(N)}, 0, \dots, 0 \right).\]
Then by construction $A_N = A_N' + A_N''$ and $B_N = B_N' + B_N''$ and that since the $|\theta_j - a_j| \leq \varepsilon$ and $||A_N''||_{op} =  \displaystyle \max_{j \in \{1, \dots, k(N)\}} |\theta_j-a_j|$, we have $||A_N''||_{op} \leq \varepsilon$. Similary $||B_N''||_{op} \leq \varepsilon$, hence $||A_N'' + B_N''||_{op} \leq 2\varepsilon$. 

Since we have that for all $N\geq N_0$, both $\sigma(A_N') \subset \text{Supp}(\mu)$ and $ \sigma(B_N') \subset \text{Supp}(\nu)$, Theorem \ref{cm} tells us that for any $\eta >0$, there exists a $N_1 \in \N$ such that $\sigma(A_N'+B_N') \subseteq \text{Supp}(\mu\boxplus\nu)_{\eta}$ for all $N \geq N_1$.

%

Notice $\sigma(A_N + B_N) = \sigma(A_N' + B_N' + (A_N'' + B_N''))$, we are perturbing the matrix $A_N' + B_N'$ by $A_N'' + B_N''$, and $ || A_N'' + B_N''|| \leq 2\varepsilon$. The $\varepsilon$-pseudospectrum gives a nice description of what can happen to the spectrum under such a perturbation. 

\begin{definition}
The \textit{pseudeospectrum} (more specifically the \textit{$\varepsilon$-pseudospectrum}) of a square matrix $A$ is defined as
\[\sigma_{\varepsilon}(A) = \{ \lambda \in \C ~|~ A-\lambda I ~\text{ has least singular value at most } \varepsilon \}.\]
\end{definition}

The following are two equivalent definitions for pseudospectrum. 
\[\sigma_{\varepsilon}(A) = \{\lambda \in \C ~|~ \exists \text{ unit vector } v \text{ s.t. } |(A-\lambda I)v| \leq \varepsilon \}\]
\[\sigma_{\varepsilon}(A) = \displaystyle \sigma(A) \cup \left\{\lambda \in \C ~|~ ||(A-\lambda I)\iv|| > \frac{1}{\varepsilon}\right\}.\]

The $\varepsilon$-pseudospectrum $\sigma_{\varepsilon}(A)$ describes how $\sigma(A)$ can change under small pertubations, that is,  pertubations of type $A+B$ where $B$ has norm at most $\epsilon$. 

\begin{prop}\label{pseudo}
We have the following properties of the pseudospectrum: 
\begin{enumerate}
\item $\lambda \in \sigma_{\varepsilon}(A)$ if and only if $\lambda \in \sigma(A+B)$ where for some $B$ with $||B|| \leq \varepsilon$.
\item $\sigma_{\varepsilon}(A)$ contains the $\varepsilon$-neighborhood of $\sigma(A)$.
\item If $A$ is normal, then $\sigma_{\varepsilon}(A)$ is exactly the $\varepsilon$-neighborhood of $\sigma(A)$. 
\end{enumerate}
\end{prop}

The proof of this proposition uses straightforward techniques from linear algebra. More about the pseudospectrum can be found in \cite{hogben2006handbook}. 

We can now conclude the proof of Theorem \ref{conv}. Recalling back to our situation, we see that we are perturbing $A_N'+B_N'$ by the matrix $A_N'' + B_N''$ which has operator norm at most $2 \varepsilon$. Hence, for all $N \geq \max\{ N_0,N_1\}$ we have

\begin{align*}
\sigma(A_N+B_N) &= \sigma(A_N'+B_N' +(A_N'' + B_N'')) &\\
&\subseteq \sigma_{2\varepsilon}(A_N' + B_N')  &\text{by Proposition \ref{pseudo} (1)}\\
&= \sigma(A_N' + B_N')_{2\varepsilon}  &\text{by Proposition \ref{pseudo} (3)}\\
&\subseteq \text{Supp}(\mu \boxplus \nu)_{2\varepsilon+\eta}  &\text{by Theorem \ref{cm}}.
\end{align*}
\end{proof}

Notice that the result of Theorem \ref{conv} can be simplified to \textit{there exists an $N_1\in \N$ such that for all $N \geq N_1$, we have $\sigma(A_N+B_N) \subseteq \text{Supp}(\mu\boxplus\nu)_{3\varepsilon }$}, since we can set $\eta = \varepsilon$. Section 5 contains the statements and proofs for a multiplicative version of  Theorem \ref{conv} for measures on $\R_{>0}$ and $\T$.

~\\
In an effort to make the paper self-contained, we present a number of lemmas which come from \cite{belinschi2017outliers}. We offer proofs where convenient.

\begin{lemma}\label{4.5}
Let $\gamma = \R$ or $\mathbb{T}$, and let $K \subsetneq \gamma$  be compact, and let $p$ be a positive integer. Consider an analytic function $F:\overline{C}\bs K \to M_p(\C)$ such that $F(z)$ is diagonal for each $z \in \C \bs K$, $F(\infty) = I_p$, and $z \mapsto (F(z))_{ii} \in \C$ has only simple zeros, all of which are contained in $\gamma \bs K$, $1 \leq i \leq p$. Fix $\delta >0$ such that $\det(F)$ has no zeros on the boundary of $K_\delta$ relative to $\R$, and let $\rho_1, \dots, \rho_s \in \gamma$ be a list of those points $z\in \C \bs K_\delta$ for which $F(z)$ is not invertible. 

Suppose that there exists positive numbers $\{\delta_N\}_{N\in \N}$ and analytic maps $F_N: \overline{C}\bs K_{\delta_N} \to M_p(\C)$, $N \in \N$, such that:
\begin{enumerate}[\indent $1$.]
\item $\lim_{N \to \infty} \delta_N = 0$;
\item $F_N(z)$ is invertible for $z \in \C \bs \gamma$ and $N \in \N$; and 
\item $F_N$ converges to $F$ uniformly on compact sets of $\C\bs K$.
\end{enumerate}
Then,
\begin{enumerate}[\indent i.]
\item $\dim(\ker(F(\rho_j)))$ equals the order of $\rho_j$ as a zero of $z \mapsto \det(F(z))$;
\item Given $\varepsilon >0$ such that
\[\varepsilon < \frac{1}{2} \min\{|\rho_i - \rho_j|, \dist(\rho_i, K_\delta): 1 \leq i \neq j \leq s\},\]
there exists an integer $N_0$ such that for $N \geq N_0$, we have;
\begin{enumerate}[\indent {$-$}]
\item counting multiplicities, $\det(F_N)$ has exactly $\dim(\ker(F(\rho_j)))$ zeros in $(\rho_j-\varepsilon, \rho_j + \varepsilon) \subseteq \gamma$, $j = 1, \dots, s$, and 
\item $\{z \in \C \bs K_{\delta} : ~ \det(F_N(z)) = 0\} \subset \bigcup_{j=1}^{s} (\rho_j - \varepsilon, \rho_j+\varepsilon)$. 
\end{enumerate}
\end{enumerate}
\end{lemma}

\begin{proof}
Assertion $(i)$ is obvious. By assumption $(3)$ we have that $F_N$ converges to $F$ uniformly on compact sets of $\overline{\C} \bs K$, then it follows that the funcitons $f_N (z) = det(F_N(z))$ converge to $f(z) = det(F)$ on compact sets of $\overline{\C} \bs K$. Then by Hurwitz's theorem \cite{gamelin2003complex}, we have that for sufficiently large $N$, $f_n$ has exactly as many zeros as $f$ in $\overline{\C} \bs K$, counting multiplicites. Since we assumed that all the zeros of $f_N$ are assumed to be in $\gamma$ and therefore the zeros cluster around $\{\rho_1, \dots, \rho_s\}$ in the following sense: for any given $\varepsilon > 0$, there exists an $N_\varepsilon \in \N$ such that 
\[\{z \in \C \bs K_{\delta} ~:~ \det(F_N(z)) = 0\} \subset \bigcup_{j=1}^{s}(\rho_j - \varepsilon, \rho_j+\varepsilon )\]
when $N \geq N_{\varepsilon}$. When $\varepsilon >0$ is small enough, there are exactly $\dim(\ker(F(\rho_j)))$ zeros of $f_N$ in $(\rho_j-\varepsilon, \rho_j+\varepsilon)$, counting multiplicities. 
\end{proof}

For the next lemma we provide the following notation. If $X \in M_m(\C)$ is a normal matrix, we denote $E_X$ its spectral measure, and if $S \in \C$ is a Borel set, then $E_X(S)$ is the orthogonal projection onto the linear span of all eigenvectors of $X$ corresponding to eigenvalues in $S$. 

\begin{lemma}\label{4.10}
Let $X$ and $X_0$ be Hermitian $N \times N$ matrices. Assume that $\alpha, \beta, \delta \in \R$ are such that $\alpha < \beta$, $\delta >0$, and neither $X$ nor $X_0$ has any eigenvalues in $[\alpha-\delta, \alpha] \cup [\beta, \beta+\delta]$. Then
\[|| E_x((\alpha, \beta)) - E_{X_0}((\alpha, \beta))|| < \frac{4(\beta - \alpha + 2\delta)}{\pi \delta^2}||X-X_0||.\]
In particular, for any unit vector $\xi \in E_{X_0}((\alpha,\beta))(\C^N)$,
\[|| I_N - E_{X}((\alpha, \beta))\xi|| < \frac{4(\beta - \alpha + 2\delta)}{\pi \delta^2}||X-X_0||.\]
\end{lemma}

\begin{proof}
Let $\gamma$ be the rectangular path in $\C$ with corners at the points $(\alpha - 1) \pm \frac{\delta}{2}i$ and $(\beta - 1) \pm \frac{\delta}{2}i$. By assumption, we have $\sigma(X) \cap ([\alpha-\delta, \alpha] \cup [\beta, \beta+\delta]) = \emptyset$ and $\sigma(X_0) \cap ([\alpha-\delta, \alpha] \cup [\beta, \beta+\delta]) = \emptyset$. Thus we can obtain the spectral projections $E_X(\alpha,\beta)$ and $E_{X_0}(\alpha,\beta)$ by the analytic functional calculus:
\[E_X((\alpha, \beta)) - E_{X_0}((\alpha, \beta)) = \frac{1}{2\pi} \int_{\gamma} [(\lambda - X)\iv - (\lambda - X_0)\iv] d\lambda.\]
Thus we have the following norm estimate
\begin{align*}
||E_X((\alpha, \beta)) - E_{X_0}((\alpha, \beta)) || & = \frac{1}{2\pi}\left| \left|    \int_{\gamma} (\lambda - X)\iv (X_0 - X)(\lambda - X_0)\iv d \lambda  \right|\right| \\
&\leq \int_{\gamma} ||(\lambda - X)\iv (X_0 - X)(\lambda - X_0)\iv|| d \lambda\\
&\leq (\beta - \alpha + 2 \delta)\frac{||X-X_0||}{\pi}\sup_{\lambda \in \gamma}\frac{1}{|| \lambda - X||}\sup_{\lambda \in \gamma}\frac{1}{|| \lambda - X_0||}\\
&<\frac{4(\beta-\alpha + 2\delta)}{\pi \delta^2}||X-X_0||.
\end{align*}
\end{proof}

\begin{lemma}\label{RandER}
Fix a positive integer $r$, a projection $P$ of rank $r$ and a scalar $z \in \C \bs \R$. Then 
\[\lim_{N\to \infty} ||PR_N(z)P^* - P\E[R_N(z)]P^*|| = 0, ~~\text{ almost surely. }\]
\end{lemma}

\begin{proof}
The claim is equivalent to the statement that, given unit vectors $h,k \in \C^N$,
\begin{equation}\label{*}
\lim_{N\to \infty}k^*(R_N(z) - \E[R_N(z)])h = 0
\end{equation}
almost surely. The random variable $k^*R_N(z)h$ is a Lipschitz function on the unitary group U$(N)$ with Lipschitz constant $C/|\mathfrak{I}z|^2$. An application of \cite{anderson2010introduction}, Corollary 4.4.28, yields the inequality
\[\mathbb{P} \left( |k^*(R_N(z) - \E[R_N(z)])h| > \frac{\epsilon}{N^{\frac{1}{2}-\alpha}}\right) \leq 2\text{exp}(-CN^{2\alpha}|\mathfrak{I}z|^4\epsilon^2)\]
for any $\alpha \in (0, 1/2)$, and (\ref{*}) follows by an application of the Borel-Cantelli lemma. 
\end{proof}

\begin{lemma}\label{bbcf5.1}
Fix a positive integer $p$, and let $C_N$ and $D_N$ be deterministic real diagonal $N\times N$ matrices whose norms are uniformly bounded and such that the limits 
\[n_i = \lim_{N\to \infty} (C_N)_{ii}\]
exist for all $i = 1, 2, \dots, p$. Suppose that the empirical eigenvalue distributions of $C_N$ and $D_N$ converge weakly to $\mu$ and $\nu$, respectively. Then the resolvent
\[R_N(z) = (zI_N -C_N - U_NDU_N^*)\iv, ~~ z \in \C\bs \R,\]
satisfies
\[\lim_{N\to \infty} P_N\E[R_N(z)]P_N^* = \emph{\diag}\left( \frac{1}{\omega_1(z)-\eta_1}, \dots, \frac{1}{w_1(z)-\eta_p}\right).\]
\end{lemma}

\begin{lemma}\label{bbcf5.1-mult}
Fix a positive integer $p$, and let $C_N$ and $D_N$ be deterministic nonnegative diagonal $N \times N$ matrices with uniformly bounded norms such that, for all $i = 1, \dots, p$, $(C_N)_{ii} \neq 0$ and the limits
\[\eta_i = \lim_{N\to \infty} (C_N)_{ii}\]
exist. Suppose that the empirical eigenvalue distributions of $C_N$ and $D_N$ converge weakly to $\mu$ and $\nu$, respectively. Then the resolvent
\[R_N(z) = (zI_N - C_N^{1/2}U_ND_NU_N^*C_N^{1/2})^{-1}, ~~ z\in\C \bs \R\]
satisfies
\[\lim_{N\to \infty} P_N\E[zR_N(z)]P_N^* = \emph{\diag}\left( \frac{1}{1-\eta_1 \omega_1(z\iv)}, \dots, \frac{1}{1-\eta_p \omega_1(z\iv)}\right).\]
\end{lemma}


\section{Proof of the Main Results}
	
Notice that it is sufficient to prove Theorem \ref{main} for deterministic matrices $A_N$ and $B_N$. If $A_N$ and $B_N$ are independent, then we may choose the underlying probability space $\Omega$ to be of the form $\Omega = \Omega_1 \times \Omega_2$ where $A_N$ is a measurable function on $\Omega_1$ and $B_N$ is a measurable function on $\Omega_2$. Let $\tilde{\Omega}$ be the event that (\ref{cond1}) and (\ref{cond2}) hold, where
\begin{align}
&\text{ $-$ there exists an $N_0 \in \N$ such that for all $N >N_0$, we have $\sigma(X_N) \subset K'_{\varepsilon}$ }; \label{cond1}\\
&\text{ $-$ there exists an $N_0 \in \N$ such that for all $N>N_0$,we have} \label{cond2}
\end{align}
\[\card(\{\sigma(X_N)\cap (\rho-\varepsilon,\rho +\varepsilon)\}) = k+\ell.\]
The event $\tilde{\Omega} \in \Omega$ is a measurable set. Denote a point in $\Omega_1 \times \Omega_2$ as $(w_1,w_2)$ (we use $w$'s in place of $\omega$'s in order to distinguish them from subordination functions).  Assume the theorem holds for deterministic matrices. Then for almost all $w_2 \in \Omega_2$, there exists a set $\tilde{\Omega}_1(w_2)$ such that for all $w_1 \in \tilde{\Omega}_1$, (\ref{cond1}) and (\ref{cond2}) hold for $X_N(w_1, w_2) = A_N(w_1) + U_N B_N(w_2)U_N^*$. The set of all such points $(w_1, w_2)$ has outer measure one and contained in $\tilde{\Omega}$, hence $\tilde{\Omega}$ has measure one by Fubini's theorem. 
$~$\\

Our proof largely mimics the proof in section 5 of \cite{belinschi2017outliers} with slight adjustments. Due to the left and right invariance of the Haar measure on $\text{U}(N)$, we may assume without loss of generality that both $A_N$ and $B_N$ are diagonal matrices. Let 
\[A_N = \displaystyle \diag\left(\theta_1,\dots, \theta_{\varphi(N)},\alpha_1^{(N)}, \dots, \alpha_{N-\varphi(N)}^{(N)}\right)\]
and
\[B_N = \displaystyle \diag\left(\tau_1,\dots, \tau_{\psi(N)},\beta_1^{(N)}, \dots, \beta_{N-\psi(N)}^{(N)}\right)\]
with $\alpha_1^{(N)} \geq \dots \geq \alpha_{N-\varphi(N)}^{(N)}$ and no order relations between the $\theta_j$'s and $\alpha_{j}^{(N)}$'s except $\theta_j \neq \alpha_j^{(N)}$, and similar order relations for the $\beta$'s and $\tau$'s.  
~\\

Let $\varepsilon > 0$ and, let $\theta_{i_1}, \dots, \theta_{i_p}$ be the $p$ elements of $\{\theta_i\}_{i=1}^{\infty}$ that lie outside $\supp(\mu)_{\varepsilon/2}$. Similarly, let $\tau_{j_1}, \dots, \tau_{j_q}$ be the $q$ elements of $\{\tau_j\}_{j=1}^{\infty}$ that lie outside $\supp(\nu)_{\varepsilon/2}$. We know that $p,q < \infty$ since dist$(\theta_i, \text{Supp}(\mu)) \to 0$ as $i \to \infty$ and dist$(\tau_j, \text{Supp}(\nu)) \to 0$ as $j \to \infty$.  

Let $M_0\in \N$ be large enough such that for all $N \geq M_0$, we have that $\sigma(A_N)\bs \supp(\mu)_{\varepsilon/2} = \{\theta_{i_1}, \dots, \theta_{i_p}\}$. For $N \geq M_0$, we may reorder the sequence $(\theta_i)_{i \in \N}$ to write 
\[A_N = \displaystyle \diag\left(\theta_1,\dots,\theta_p, \theta_{p+1}, \dots, \theta_{\varphi(N)},\alpha_1^{(N)}, \dots, \alpha_{N-\varphi(N)}^{(N)}\right)\]
where $\theta_1,\dots,\theta_p$ are precisely the $p$ elements of $\{\lambda_n(A_N)\}_{n=1}^{N}$ that eventually (as $N \to \infty$) lie outside $\supp(\mu)_{\varepsilon/2}$.
Let $\alpha \in \supp(\mu)$ and define the following
\[A_N' = \diag(\underbrace{\alpha,\dots,\alpha}_{p\text{-times}}, \theta_{p+1}, \dots, \theta_{\varphi(N)},\alpha_1^{(N)}, \dots, \alpha_{N-\varphi(N)}^{(N)})\]
and
\[A_N'' = \diag(\theta_1-\alpha,\dots,\theta_p-\alpha, \underbrace{0, \dots, 0}_{(N-p)\text{-times}})\]
Hence $A_N = A'_N + A''_N$, and $A''_N = P_N^* \Theta P_N$ where $P_N: \C^N \lra \C^p$ is the projection onto the first $p$ coordinates and $\Theta = \diag(\theta_1-\alpha,\dots,\theta_p-\alpha)$. 

Let $M_1 \in \N$ be large enough such that for all $N \geq M_1$ we have $\sigma(B_N) \bs \supp(\nu)_{\varepsilon/2} = \{\tau_{j_1}, \dots, \tau_{j_q}\}$. Like above, when $N \geq M_1$ we may reorder the sequence $(\tau_j)_{j\in \N}$ to write
\[B_N = \displaystyle \diag\left(\tau_1,\dots,\tau_q, \tau_{q+1}, \dots, \tau_{\psi(N)},\beta_1^{(N)}, \dots, \beta_{N-\psi(N)}^{(N)}\right)\]
where $\tau_1, \dots, \tau_q$ are precisely the $q$ elements of $\{\lambda_n(B_N)\}_{n=1}^{N}$ that eventually lie outside $\supp(\nu)_{\varepsilon/2}$. Let $\beta \in \supp(\nu)$ and define the following, 
\[B_N' = \diag(\underbrace{\beta,\dots,\beta}_{q\text{-times}}, \tau_{q+1}, \dots, \tau_{\psi(N)},\beta_1^{(N)}, \dots, \beta_{N-\psi(N)}^{(N)})\]
and
\[B_N'' = \diag(\tau_1-\beta,\dots,\tau_q-\beta, \underbrace{0, \dots, 0}_{(N-q)\text{-times}})\]
Thus we have $B_N = B_N' + B_N''$. We can also express $B_N''$ as $B_N'' = Q_N^*TQ_N$ where $Q_N: \C^N \lra \C^q$ is the projection onto the first $q$ coordinates, and $T = \diag(\tau_1-\beta, \dots, \tau_q-\beta)$. 

We follow the strategy of \cite{belinschi2017outliers} and reduce our problem to that of a $p \times p$ matrix. Define $X_N' = A_N' + U_NB_N'U_N^*$ and let $N_0 = \max\{M_0, M_1\}$. By construction of $N_0$, we have that $\sigma(A'_N) \subseteq \supp(\mu)_{\varepsilon/2}$ for all $N>N_{0}$, and $\sigma(B'_N) \subseteq \supp(\nu)_{\varepsilon/2}$ for all $N>N_{0}$. Hence by Theorem \ref{conv}, there exist positive random variables $(\delta_{N})_{N\in N}$ such that $\lim_{N\to \infty} \delta_N = 0$ almost surely, such that

\[\sigma(A_N' + U_NB'_NU_N^*) \subseteq (\supp(\mu \boxplus \nu)_{\varepsilon})_{\delta_N} = \supp(\mu \boxplus \nu)_{\varepsilon+\delta_N} = K_{\varepsilon+\delta_N}. \]

Let $z \in \C \bs K_{\varepsilon+\delta_N}$, we have
\begin{align*}
zI_N - (A_N + U_NB'_NU_N^*) &= zI_N - X_N' - A_N'' \\
&= (zI_N - X_n')(I_N - (zI_N - X_N')\iv A_N'')
\end{align*}
and therefore

\[\det(zI_N - (A_N + U_NB'_NU_N^*)) = \det(zI_N - X_N') \det(I_N - (zI_N - X_N')\iv P_N^*\Theta P_N).\]

Using the fact that $\det(I-XY) = \det(I - YX)$ when $XY$ and $YX$ are square matrices we get,
\[\det(I_N - (zI_N - X_N')\iv P_N^*\Theta P_N) = \det(I_p - P_N(zI_N - X_N')\iv P_N^*\Theta).\]
Thus
\[\det(zI_N - (A_N + U_NB'_NU_N^*)) = \det(zI_N - X_N') \det(I_p - P_N(zI_N - X_N')\iv P_N^*\Theta),\]
and hence we conclude that the eigenvalues of $A_N +U_NB'_NU_N^*$ outside $K_{\varepsilon+\delta_N}$ are precisely the zeros of the function $\det(F_N(z))$ where 

\begin{equation}\label{F_N}
F_N(z) = I_p -P_N(zI_N - X_N')\iv P_N^* \Theta 
\end{equation}
which is a random analytic function defined on $\overline{\C} \bs K_{\varepsilon+\delta_N}$ with values in $M_p(\C)$. Our next step is to show that $\{F_N(z)\}_N$ converges almost surely to the deterministic diagonal matrix function

\begin{equation}\label{F}
\displaystyle F(z) = \diag \left( 1 - \frac{\theta_1- \alpha}{\omega_1(z)-\alpha}, \dots, 1 - \frac{\theta_p- \alpha}{\omega_1(z)-\alpha} \right).
\end{equation}
$~$\\

Notice that $A_N$ and $B_N$ are deterministic real diagonal $N\times N$ matrices whose norms are uniformly bounded. Also, notice that the limits $\lim_{N\to \infty} (A'_N)_{ii}$ exists for all $i = 1, \dots, p$, in particular, $\lim_{N\to \infty} (A'_N)_{ii} = \alpha$. Hence we can apply Lemma \ref{bbcf5.1}, which says that for $z \in \C \bs \R$, the resolvent $R_N(z) = (zI_n - A'_N + U_NB_NU_N^*)\iv$ satisfies 

\begin{equation}\label{exp}
\displaystyle \lim_{N \to \infty} P_N \mathbb{E}[R_N(z)]P_N^* = \diag \left( \frac{1}{\omega_1(z) - \alpha}, \dots, \frac{1}{\omega_1(z) - \alpha}\right).
\end{equation}
$~$\\

\begin{prop}\label{FNtoF} 
Almost surely, the sequence $\{F_N\}_{N}$ converges uniformly on compact subsets of $\overline{\C} \bs K_{\varepsilon}$ to the analytic function $F$ defined by 
\[\displaystyle F(z) = \diag \left( 1 - \frac{\theta_1- \alpha}{\omega_1(z)-\alpha}, \dots, 1 - \frac{\theta_p- \alpha}{\omega_1(z)-\alpha} \right), ~~~~ z \in \overline{\C} \bs K_{\varepsilon}.\]

\end{prop}
\begin{proof}
Recall that $\alpha \in \supp(\mu)$, and that by a property of the subordination function $\omega_1$, we have that if $x \in \R \bs \supp(\mu \boxplus \nu)$ then $\omega_1(x) \in (\R \cup \{\infty\}) \bs \supp(\mu)$, hence $\omega_1(z) \neq \alpha$ for any $z \in \overline{\C} \bs K_{\varepsilon}$ (Lemma 3.1 of \cite{belinschi2017outliers}). Therefore the function $z \mapsto 1/(\omega_1(z) - \alpha)$ is analytic on  $\overline{\C} \bs K_{\varepsilon}$. Define
\[\mathcal{D} = \{ z \in \C \bs K_{\varepsilon} : \mathfrak{R}z \in \Q, \mathfrak{I}z \in \Q \bs \{0\}\}\]

The first $p$ diagonal elements of $A_N'$ are all equal to $\alpha$. Lemma \ref{RandER} gives that 
\[\lim_{N \to \infty} || P_N R_N(z) P_N^* - P_N \mathbb{E}[R_N(z)]P_N^*|| = 0 ~~ \text{ almost surely}\]
and in combination with (\ref{exp}), we get that given $z \in \mathcal{D}$, the sequence $P_N(z I_N - X_N')\iv P_N^*$ converges almost surely to $(1/(\omega_1(z) - \alpha))I_p$.  By \cite{collins2014strong}, we have that the functions are almost surely uniformly bounded on any compact set of $\C \bs K_{\varepsilon}$.

It is clear that we have uniform boundedness on some neighborhood of infinity in $\C \cup \{\infty\}$. Since $\mathcal{D}$ is dense in $\overline{\C} \bs K_{\varepsilon}$, we deduce that, almost surely, this sequence of functions $P_N(z I_N - X_N')\iv P_N^*$, converges uniformly on compact sets of $\overline{\C} \bs K_{\varepsilon}$ to the function $(1/(\omega_1 - \alpha))I_p$. Thus, 
\begin{align*}
\lim_{N\to \infty} F_N(z) &= I_p - (1/(\omega_1 - \alpha))I_p\Theta\\
&= \diag(1,\dots,1) - \diag\left(\frac{\theta_1-\alpha}{\omega_1(z) -\alpha},\dots, \frac{\theta_p-\alpha}{\omega_1(z) -\alpha}\right) \\
&=\diag \left( 1 - \frac{\theta_1- \alpha}{\omega_1(z)-\alpha}, \dots, 1 - \frac{\theta_p- \alpha}{\omega_1(z)-\alpha} \right).
\end{align*}
\end{proof}

We are now equipped ourselves with the tools to prove Theorem \ref{main}. 

\begin{proof} \textit{(of Theorem \ref{main})} 

We first remark that our proof will follow a nearly identical procedure as that of Theorem 2.1 in \cite{belinschi2017outliers}. 
$~$\\

\textit{Step} 1. We first consider the case where $B_N$ has no spikes, that is, $q=0$ and $B_N = B_N'$. We work on the almost sure event on which:
\begin{enumerate}[\indent $-$]
\item there exists a random sequence $\{\delta_N\}_{N \in \N}$ such that $\lim_{N \to \infty} \delta_N = 0$ and $\sigma(A'_N + U_NB'_NU_N^*) \subseteq K_{\varepsilon + \delta_N}$ for all $N$, and 
\item the sequence $(F_N)_{N \in \N}$ defined in (\ref{F_N}) converges to the function $F$ defined by (\ref{F}) uniformly on compact sets of $\overline{\C} \bs K_{\varepsilon}$. This is guaranteed by Proposition \ref{FNtoF}.
\end{enumerate}

We apply Lemma \ref{4.5} on this event with $\gamma = \R$ and $K = K_{\varepsilon}$. We  want to keep the same result when we exchange $K = \supp(\mu \boxplus \nu)$ with $K_{\varepsilon}$ for small $\varepsilon$. For the case where $\gamma = \R$, notice that $K_{\varepsilon}\subsetneq \gamma$ is compact. Our function $F$ is an analytic function from $\overline{\C}\bs K_{\varepsilon} \lra M_p(\C)$ such that $F(z)$ is diagonal for each $z \in \C \bs K_{\varepsilon}$ and $F(\infty) = I_p$. Notice that the zeros of the map $z \mapsto (F(z))_{ii} \in \C$ is where $\omega_1(z) = \theta_i$, and all contained in $\gamma \bs K_{\varepsilon}, $ for $1 \leq i \leq p$. Hence, these assumptions are satisfied with $K_{\varepsilon}$ in place of $K$ and we can use this lemma freely, as long as the remaining conditions are met.  

To show that the zeros of $(F(z))_{ii}$ are simple, we use an application of the Julia-Carath\'eodory theorem (\cite{garnett2007bounded}, Chapter I, Exercises 6 and 7). Conditions (1) and (3) of Lemma \ref{4.5} are satisfied due to Proposition \ref{FNtoF}. For condition (2), we see that if $F_N(z)$ were not invertible then $\det(F_N(z)) = 0$. And since all zeros of $\det(F_N(z))$ are the eigenvalues of $A_N + U_NB'_NU_N^*$ which is a self-adjoint matrix, we see that $z$ must be real, and condition (2) is satisfied. Lastly, there are arbitrarily small numbers $\delta>0$ such that the boundary points of $K_{\varepsilon+\delta}$ are not zeros of $\det{(F)}$, thus all the conditions of Lemma \ref{4.5} are satisfied, the consequences of which provide precisely the results of Theorem \ref{main} for the case when $q = 0$. Namely, the eigenvalues of $A_N + U_NB'_NU_N^*$ in $K_{\epsilon + \delta}$ are precisely the zeros of $\det(F_N)$, and the set of points $z$ such that $F(z)$ is not invertible are $\bigcup_{i=1}^{p}\omega_1^{-1}(\{\theta_i\})$. 
$~$\\

\textit{Step} 2. Now we suppose that $q >0$ and $k = 0$. By step 1, we know that there exist random variables $(\delta_N)_{N\in \N}$ such that $\lim_{N\to \infty} \delta_N = 0$ almost surely and $\sigma(A_N + U_NB'_NU_N^*) \subseteq (K''_{\varepsilon})_{\delta_N}$ where 
\[K''_{\varepsilon} = K_\varepsilon \cup \left[ \bigcup_{i=1}^{\infty} \omega_1\iv(\{\theta_i\})\right].\]

We now switch the roles of $A_N$ and $B_N$ and proceed as in Step 1. With the reasoning as above, we see that the eigenvalues of $X_N = A_N +U_NB_NU_N^*$ outside of $(K''_{\varepsilon})_{\delta_N}$ are precisely the zeros of the function $\det(\tilde{F}_N(z))$, where 
\[\tilde{F}_N(z) = I_q - Q_N(zI_N - U_NA_NU_N^* - B'_N)\iv Q_N^*T.\]
We now apply Lemma \ref{4.5} to the functions $\tilde{F}_N$ and $\tilde{F}$, where 
\[\tilde{F}(z) = \diag \left( 1-\frac{\tau_1-\beta}{\omega_2(z) - \beta},\dots, 1-\frac{\tau_q-\beta}{\omega_2(z) - \beta} \right)\]
and the compact set $K''_{\varepsilon}$. We can use a simply modified version Proposition \ref{FNtoF} to conclude that $\{\tilde{F}_N\}_{N \in \N}$ converges to $\tilde{F}$. This concludes Step 2 and by symmetry we have also proved the case when $q>0$ and $l = 0$. 
$~$\\

\textit{Step} 3. Lastly, we consider the case when $q >0$ and $\ell \cdot k >0$, that is, both $k$ and $\ell$ cannot be zero. For this case, we us a perturbation argument and apply Lemma \ref{4.10}. Fix $\rho \in \R \bs K_{\varepsilon}$, such that $\omega_1(\rho) = \theta_{i_0}$ for some $i_0 \in \{1, \dots, p\}$ and $\omega_2(\rho) = \tau_{j_0}$ for some $j_0 \in \{1, \dots, q\}$. Let $\xi <0$ such that such that $(\rho - 2\xi, \rho + 2\xi) \cap K' = \{ \rho \}$. Choose $\delta \in (0, \xi/3)$ small enough that $\omega_1((\rho - 3\delta, \rho + 3\delta))$ contains no spikes $\theta_i = \theta_{i_0}$ and $\omega_2((\rho - 3\delta, \rho + 3\delta))$ contains no spikes $\tau_j = \tau_{i_0}$. Since $\omega_1$ is strictly increasing on $(\rho - 3\delta, \rho + 3\delta)$, we have $\omega_1(\rho + 2 \delta) = \theta_{i_0} + \eta$ with $\eta >0$. If we consider the perturbed model 
\[X_N' = X_N + \eta E_{A_N}(\{\theta_{i_0}\}),\]
then we can use step 2 to conclude that, almost surely for large $N$, $X_N'$ has $\ell$ eigenvalues in $(\rho - \delta, \rho + \delta)$ and $k$ eigenvalues in the disjoint interval $(\rho+\delta, \rho +3\delta)$. Thus neither $X_N$ nor $X_N'$ have eigenvalues in the interval set $[(\rho-\delta)- (\xi - 3\delta), (\rho - \delta)]\cup[(\rho + 3\delta), (\rho + 3\delta)+(\xi -3 \delta)]$. Apply Lemma \ref{4.10} on $X_N$ and $X_N'$ with respect to this set and we get
\begin{align*}
|| E_{X_N}((\rho - \delta,\rho + 3\delta)) - E_{X_N'}((\rho - \delta,\rho + 3\delta))||  &< \frac{4((\rho - 3\delta) - (\rho - \delta) + 2(\xi - 3\delta))}{\pi(\xi - 3\delta)^2}||X_N - X_N'||\\
&<\frac{8(\xi - \delta)}{\pi(\xi - 3\delta)^2}\eta.
\end{align*}
Notice that as $\delta \to 0$, we get that $\eta \to 0$, hence
\[\lim_{\delta \to 0} || E_{X_N}((\rho - \delta,\rho + 3\delta)) - E_{X_N'}((\rho - \delta,\rho + 3\delta))|| = 0.\]
Thus, since $X_N'$ has $k+\ell$ eigenvalues in $(\rho - \delta, \rho +3\delta)$, we have that $X_N$ has $k + \ell $ eigenvalues in $(\rho - \xi, \rho +\xi)$.

\end{proof}


\section{Multiplicative Cases}
	Similar results hold for free multiplicative convolution both on the positive real line and on the unit circle. We first multiplicative model we consider is $X_N = A_N^{1/2}U_NB_NU_N^*A_N^{1/2}$, where $\mu_{A_N}$ converges weakly to a measure $\mu$ such that $\supp(\mu) \in [0,\infty)$, and $\mu_{B_N}$ converges weakly to a measure $\nu$ such that $\supp(\mu) \in [0,\infty)$. We know from \cite{voiculescu1991limit} that the empirical eigenvalue distribution $\mu_{X_N}$ converges weakly to $\mu \boxtimes \nu$. 

Before we state and prove a multiplicative analogue of Theorem \ref{main}, we prove a multiplicative analogue to Theorem \ref{conv}. 
\begin{lemma}\label{convMultPos}
Let $A_N$ and $B_N$ are $N \times N$ independent Hermitian random matrices such that the laws of $A_N$ and $B_N$ are unitarily invariant. Let $\mu$ and $\nu$ be compactly supported measures on $[0,\infty)$ such that almost surely $\mu_{A_N} \lra \mu$ and almost surely $\mu_{B_N} \lra \nu$. Let $\varepsilon > 0$ be given. Suppose there exists an $N_0 \in \N$ large enough such that,  for $N \geq N_0$, we have both
\[\sigma(A_N) \subseteq \text{\emph{Supp}}(\mu)_{\varepsilon}\]
\[ \sigma(B_N) \subseteq \text{\emph{Supp}}(\nu)_{\varepsilon}.\]
Then for $\eta>0$ there exists an $N_1$ such that for all $N \geq N_1$ and a constant $C >0$ such that $\sigma(A_N^{1/2}B_NA_N^{1/2}) \subseteq  \text{\emph{Supp}}(\mu \boxtimes \nu)_{\eta +C\varepsilon}$.
\end{lemma}
\begin{proof}
Suppose $N > N_0$. Due to our assumption of unitary invariance, we may assume $A_N$ and $B_N$ are diagonal matrices. We write
\[A_N = \diag(\theta_1, \theta_2, \dots, \theta_{k(N)}, \alpha_1, \dots, \alpha_{N-k(N)})\]
\[B_N = \diag(\tau_1, \tau_2, \dots, \tau_{\ell(N)}, \beta_1, \dots, \beta_{N-\ell(N)}).\]
where $\theta_1, \dots, \theta_{k(N)}$ are the $k(N)$ elements in $(\lambda_{i}(A_N))_{i=1}^{N}$ that are outside the $\supp(\mu)_\varepsilon$, and $\tau_1, \dots, \tau_{\ell(N)}$ are the $\ell(N)$ elements in $(\lambda_{i}(B_N))_{i=1}^{N}$ that are outside the $\supp(\nu)_\varepsilon$. Define $a_i$ as a point in $\supp(\mu)$ such that $|\theta_i - a_i| = \dist(\theta_i, \supp(\mu))$, and similarly define $b_i$ as a point in $\supp(\nu)$ such that $|\tau_i - b_i| = \dist(\tau_i, \supp(\nu))$. Define
\[A'_N = \diag(a_1, \dots, a_{k(N)}, \alpha_1, \dots, \alpha_{N-k(N)})\]
\[B'_N = \diag(b_1, \dots, b_{k(N)}, \beta_1, \dots, \beta_{N-\ell(N)}).\]

Since $\sigma(A'_N) \in \supp(\mu)$ and $\sigma(B'_N) \in \supp(\nu)$ for all $N \geq N_0$, we know by Theorem \ref{cm} that for any  $\eta >0$ there exists an $N_1$ such that for all $N \geq N_1$ we have $\sigma(A_N^{'1/2}B'_NA_N^{'1/2}) \subseteq \supp(\mu \boxtimes \nu)_\eta$. From \cite{tao2012topics}, we see that for any $i = 1, \dots, N$ we have 
\[|\lambda_i(A_N^{1/2}B_NA_N^{1/2}) - \lambda_i((A'_N)^{1/2}B'_N(A'_N)^{1/2})| \leq ||A_N^{1/2}B_NA_N^{1/2} - (A'_N)^{1/2}B'_N(A'_N)^{1/2}||\]
and we can make the following estimate,
\begin{align*}
\\||A_N^{1/2}B_NA_N^{1/2} - (A'_N)^{1/2}B'_N(A'_N)^{1/2}||   &= ||A_N^{1/2}B_NA_N^{1/2} - (A'_N)^{1/2}B_NA_N^{1/2} \\
&\indent + (A'_N)^{1/2}B_NA_N^{1/2} - (A'_N)^{1/2}B'_NA_N^{1/2}\\
&\indent + (A'_N)^{1/2}B'_NA_N^{1/2} - (A'_N)^{1/2}B'_N(A'_N)^{1/2}|| \\
& \leq ||A_N^{1/2}B_NA_N^{1/2} - (A'_N)^{1/2}B_NA_N^{1/2}|| \\
&\indent + ||A_N^{'1/2}B_NA_N^{1/2} - (A'_N)^{1/2}B'_NA_N^{1/2}|| \\
&\indent + ||(A'_N)^{1/2}B'_NA_N^{1/2} - (A'_N)^{1/2}B'_N(A'_N)^{1/2}|| \\
& \leq ||B_NA_N^{1/2} ||\cdot|| A_N^{1/2} - (A'_N)^{1/2} || \\
&\indent + || (A'_N)^{1/2}||\cdot ||A_N^{1/2}|| \cdot||B_N- B'_N||\\ 
&\indent +  ||(A'_N)^{1/2}B'_N||\cdot ||A_N^{1/2} - (A'_N)^{1/2}|| \\
& \leq c_1 \varepsilon + c_2 \varepsilon +c_3 \varepsilon \\
& = C\varepsilon \\
\end{align*}

where $C, c_1, c_2, c_3$ are positive constants. Thus, it follows that for $N \geq \max\{N_1,N_0\}$ we have 
\[\sigma(A_N^{1/2}B_NA_N^{1/2}) \subseteq  \text{\emph{Supp}}(\mu \boxtimes \nu)_{\eta + C\varepsilon}.\]
\end{proof}

We now present and prove the positive multiplicity version Theorem \ref{main}.
\begin{theorem}\label{mult-real}
Suppose the following:
\begin{enumerate}[$1$.]
\item Two compactly supported Borel probability measures $\mu$ and $\nu$ on $[0,\infty)$. 

\item A sequence of fixed real numbers $\{\theta_i\}_{i=1}^{\infty}$ such that:
\begin{enumerate}[\indent {$-$}]
\item $\theta_i >0$ for all $i = 1, 2, \dots$;
\item $\theta_i$ does not belong to $\text{\emph{Supp}}(\mu)$ for all $i = 1, 2, \dots$;
\item  \text{\emph{dist}}$(\theta_i, \text{\emph{Supp}}(\mu)) \lra 0$ as $i \lra \infty$. 
\end{enumerate}

\item A sequence $\{A_N\}_{N\in\N}$ of random nonnegative matrices of size $N\times N$ such that:
\begin{enumerate}[\indent {$-$}]
\item $\mu_{A_N}$ converges weakly to $\mu$ as $N \lra \infty$;
\item a sequence $(\varphi(N))_{N \in \N}$ satisfying the conditions in  Proposition \ref{Qseq};
\item for $\theta \in \{\theta_1, \dots, \theta_{\varphi(N)}\}$, the sequence $(\lambda_n(A_N))_{n=1}^{N}$ satisfies 
\[\text{\emph{card}}\{n : \lambda_n(A_N) = \theta\}) = \text{\emph{card}}(\{i : \theta_i = \theta\});\]
\item the eigenvalues of $A_N$ which are not equal to some $\theta_i$ converge uniformly to $\text{\emph{Supp}}(\mu)$ as $N \lra \infty$, that is 
\[\lim_{N \lra \infty} \max_{\lambda_n(A_N) \notin \{\theta_1,\dots, \theta_{\varphi(N)}\}} \text{\emph{dist}}(\lambda_n(A_N),\text{\emph{Supp}}(\mu)) = 0.\]
\end{enumerate}

\item A sequence $\{U_N\}_{N \in \N}$ of unitary random matrices such that the distribution of $U_N$ is the normalized Haar measure on the unitary group U$(N)$.

\item A sequence of fixed real numbers $\{\tau_j\}_{j=1}^{\infty}$ such that:
\begin{enumerate}[\indent {$-$}]
\item $\tau_j >0$ for all $j = 1, 2, \dots$;
\item $\tau_j$ does not belong to $\text{\emph{Supp}}(\nu)$ for all $j = 1, 2, \dots$;
\item  \text{\emph{dist}}$(\tau_j, \text{\emph{Supp}}(\nu)) \lra 0$ as $j \lra \infty$. 
\end{enumerate}

\item A sequence $\{B_N\}_{N\in\N}$ of random nonnegative matrices of size $N\times N$ such that:
\begin{enumerate}[\indent {$-$}]
\item $\mu_{B_N}$ converges weakly to $\nu$ as $N \lra \infty$;
\item a sequence $(\psi(N))_{N \in \N}$ satisfying the conditions in  Proposition \ref{Qseq};
\item for $\tau \in \{\tau_1, \dots, \tau_{\psi(N)}\}$, the sequence $(\lambda_n(B_N))_{n=1}^{N}$ satisfies 
\[\text{\emph{card}}(\{n : \lambda_n(B_N) = \tau\}) = \text{\emph{card}}(\{j : \tau_j = \tau\});\]
\item the eigenvalues of $B_N$ which are not equal to some $\tau_j$ converge uniformly to $\text{\emph{Supp}}(\nu)$ as $\N \lra \infty$, that is 
\[\lim_{N \lra \infty} \max_{\lambda_n(B_N) \notin \{\theta_1,\dots, \theta_{\psi(N)}\}} \text{\emph{dist}}(\lambda_n(B_N),\text{\emph{Supp}}(\nu)) = 0.\]
\end{enumerate}
\end{enumerate}

Set $K = \text{\emph{Supp}}(\mu \boxtimes \nu)$, $v_k (z) = \omega_k(1/z)$ for $k = 1,2$, and 
\[\displaystyle K' = K \cup \left[ \bigcup_{i=1}^{\infty} v_1^{-1} (\{1/\theta_i\})\right] \cup \left[ \bigcup_{j=1}^{\infty} v_2^{-1} (\{1/\tau_j\})\right]\]
where $\omega_1,\omega_2$ are the subordination functions corresponding to the free convolution $\mu \boxtimes \nu$.
Then,
\begin{enumerate}[$1$.]
\item Given $\varepsilon >0$, almost surely, there exists an $N_0 \in \N$ such that for all $N >N_0$, we have \[\sigma(X_N) \subset K'_{\varepsilon}\] where $X_N = A_N^{1/2}U_NB_NU_N^*A_N^{1/2}$
\item  Fix a number $\rho \in K'\bs K$. Let $\varepsilon > 0 $ such that $(\rho - 2\varepsilon , \rho+2\varepsilon)\cap K' = \{\rho\}$ and set $k = \text{\emph{card}}(\{i:v_1(\rho) = 1/\theta_i\})$ and $\ell = \text{\emph{card}}(\{j:v_2(\rho) = 1/\tau_j\})$, then almost surely, there exists an $N_0 \in \N$ such that for all $N>N_0$, we have \[\text{\emph{card}}(\{\sigma(X_N)\cap (\rho-\varepsilon,\rho +\varepsilon)\}) = k+\ell.\]
\end{enumerate}
\end{theorem}

As in the last section, we may assume without loss of generality that both $A_N$ and $B_N$ are diagonal matrices. Let 
\[A_N = \displaystyle \diag\left(\theta_1,\dots, \theta_{\varphi(N)},\alpha_1^{(N)}, \dots, \alpha_{N-\varphi(N)}^{(N)}\right)\]
and
\[B_N = \displaystyle \diag\left(\tau_1,\dots, \tau_{\psi(N)},\beta_1^{(N)}, \dots, \beta_{N-\psi(N)}^{(N)}\right).\]

Let $\varepsilon > 0$ and, let $\theta_{i_1}, \dots, \theta_{i_p}$ be the $p$ elements of $\{\theta_i\}_{i=1}^{\infty}$ that eventually (as $N \to \infty$) lie outside $\supp(\mu)_{\varepsilon}$. Similarly, let $\tau_{j_1}, \dots, \tau_{j_q}$ be the $q$ elements of $\{\tau_j\}_{j=1}^{\infty}$ that eventually lie outside $\supp(\nu)_{\varepsilon}$. We know that $p,q < \infty$ since $\dist(\theta_i, \text{Supp}(\mu)) \to 0$ as $i \to \infty$ and dist$(\tau_j, \text{Supp}(\nu)) \to 0$ as $j \to \infty$.  

Let $M_0\in \N$ be large enough such that for all $N \geq M_0$, we have that $\sigma(A_N)\bs \supp(\mu)_{\varepsilon} = \{\theta_{i_1}, \dots, \theta_{i_p}\}$. For $N \geq M_0$, we may reorder the sequence $(\theta_i)_{i \in \N}$ to write 
\[A_N = \displaystyle \diag\left(\theta_1,\dots,\theta_p, \theta_{p+1}, \dots, \theta_{\varphi(N)},\alpha_1^{(N)}, \dots, \alpha_{N-\varphi(N)}^{(N)}\right)\]
where $\theta_1,\dots,\theta_p$ are precisely the $p$ elements of $\{\lambda_n(A_N)\}_{n=1}^{N}$ that eventually (as $N \to \infty$) lie outside $\supp(\mu)_{\varepsilon}$.
Let $\alpha \in \supp(\mu)\bs \{0\}$ and define the following
\[A_N' = \diag(\underbrace{\alpha,\dots,\alpha}_{p\text{-times}}, \theta_{p+1}, \dots, \theta_{\varphi(N)},\alpha_1^{(N)}, \dots, \alpha_{N-\varphi(N)}^{(N)})\]
and
\[A_N'' = \diag(\theta_1/\alpha,\dots,\theta_p/\alpha, \underbrace{1, \dots, 1}_{(N-p)\text{-times}}).\]
Hence $A_N = A'_NA''_N = A''_NA'_N$, and $A''_N = P_N^* \Theta P_N + I_N - P_N^*P_N$ where $P_N: \C^N \lra \C^p$ is the projection onto the first $p$ coordinates and $\Theta = \diag(\theta_1/\alpha,\dots,\theta_p/\alpha)$. 

Let $M_1 \in \N$ be large enough such that for all $N \geq M_1$ we have $\sigma(B_N) \bs \supp(\nu)_{\varepsilon} = \{\tau_{j_1}, \dots, \tau_{j_q}\}$. Like above, when $N \geq M_1$ we may reorder the sequence $(\tau_j)_{j\in \N}$ to write
\[B_N = \displaystyle \diag\left(\tau_1,\dots,\tau_q, \tau_{q+1}, \dots, \tau_{\psi(N)},\beta_1^{(N)}, \dots, \beta_{N-\psi(N)}^{(N)}\right)\]
where $\tau_1, \dots, \tau_q$ are precisely the $q$ elements of $\{\lambda_n(B_N)\}_{n=1}^{N}$ that eventually lie outside $\supp(\nu)_{\varepsilon}$. Let $\beta \in \supp(\nu)\bs\{0\}$ and define the following, 
\[B_N' = \diag(\underbrace{\beta,\dots,\beta}_{q\text{-times}}, \tau_{q+1}, \dots, \tau_{\psi(N)},\beta_1^{(N)}, \dots, \beta_{N-\psi(N)}^{(N)})\]
and
\[B_N'' = \diag(\tau_1/\beta,\dots,\tau_q/\beta, \underbrace{1, \dots, 1}_{(N-q)\text{-times}}).\]
Thus we similarly have $B_N = B_N' B_N''=B''_NB'_N+I_N - Q_N*Q_N$ where $Q_N: \C^N \lra \C^q$ is the projection onto the first $q$ coordinates, and $T = \diag(\tau_1/\beta, \dots, \tau_q/\beta)$. 

$~$\\
\indent We now use same technique as earlier, and reduce to a $p \times p$ matrix. We have the model $X_N = A_N^{1/2}U_NB_NU_N^*A_N^{1/2}$. Define $N_0 = \max\{M_0, M_1\}$. By construction of  $N_0$, we have that  for all $N > N_0$, both $\sigma(A'_N) \subseteq \supp(\mu)_{\varepsilon}$ and $\sigma(B'_N) \subseteq \supp(\nu)_{\varepsilon}$. By Lemma \ref{convMultPos} we know  that there exists a constant $C$ and random variables $(\delta_N)_{N \in \N}$ such that $\lim_{N \to \infty} \delta_N = 0$ almost surely, and
\[\sigma((A'_N)^{1/2}U_NB'_NU_N^*(A'_N)^{1/2}) \subset (\supp(\mu\boxtimes\nu)_{C\varepsilon})_{\delta_N} = K _{C\varepsilon + \delta_N}.\]

 Define $X'_N = (A_N)^{1/2}U_NB'_NU_N^*(A_N)^{1/2}$ and fix a $z \in \C \bs (K_{C\varepsilon + \delta_N} \cup \{0\})$ such that the matrix 
 \[zI_N - (A'_N)^{1/2}U_NB'_NU_N^*(A'_N)^{1/2}\]
is invertible. We then have
\begin{align*}
&\det(zI_N - X'_N) \\
& \indent \indent  = z^N \det(I_N - z\iv A''_N(A'_N)^{1/2}U_NB'_NU_N^*(A'_N)^{1/2} )\\
& \indent \indent = z^N \det(I_N - A''_N + A''_N(I_N - z\iv (A'_N)^{1/2}U_NB'_NU_N^*(A'_N)^{1/2} ))\\
& \indent \indent  = \det((zI_N - A''_N)(I_N - z\iv (A'_N)^{1/2}U_NB'_NU_N^*(A'_N)^{1/2} )\iv + A''_N)\\
& \indent \indent \indent \indent \times \det(zI_N - (A'_N)^{1/2}U_NB'_NU^*_N(A'_N)^{1/2}).
\end{align*}

The matrix $(zI_N - A''_N)(I_N - z\iv (A'_N)^{1/2}U_NB'_NU_N^*(A'_N)^{1/2} )\iv + A''_N$ is of the form
\[\begin{bmatrix}
    F_N(z)       & * \\
    0       & I_{N-p}
\end{bmatrix}\]
where $F_N$ is the analytic function with values in $M_p(\C)$ defined on $\C \bs (K_{C\varepsilon + \delta_N} \cup \{0\})$ by 
\begin{equation}\label{RMF}
F_N(z) := (I_p - \Theta)P_N\left(I_N - \frac{1}{z}(A'_N)^{1/2} U_N B'_N U_N^* (A'_N)^{1/2} \right)\iv P_N^* + \Theta.
\end{equation}
Since we know the matrix $zI_N - (A'_N)^{1/2}U_NB'_NU^*_N(A'_N)^{1/2}$ is invertible, we now that, for large $N$, that the nonzero eigenvalues of $X'_N$ outside of $K_{C\varepsilon + \delta_N} $ are precisely the zeros of $\det(F_N)$ in the open set $\C \bs (K_{C\varepsilon + \delta_N} \cup \{0\})$. Like in the additive case, the sequence $\{F_N\}$ converges almost surely to a diagonal $p \times p$ matrix function. Denoting $R'_N(z)$ as the resolvent of $ (A'_N)^{1/2}U_NB'_NU^*_N(A'_N)^{1/2}$, we see that 
\[F_N = (I_p - \Theta)P_N(z R'_N(z))P^*_N + \Theta.\]
Using similar techniques as we did in the proof of Proposition \ref{FNtoF}, an application of Lemma \ref{bbcf5.1-mult} with $C_N = A'_N$, $D_N = B'_N$ and $R_N'(z)$ gives the following result,
\begin{prop}\label{4.3}
Almost surely, the sequence $\{F_N\}_{N}$ converges uniformly on the compact sets of $\overline{\C} \bs K_{C\varepsilon + \delta_N} $ to the analytic function $F$ defined on $\overline{\C} \bs K_{C\varepsilon} $ by
\[F(z) = \diag\left( \left( 1-\frac{\theta_j}{\alpha}\right)\frac{1}{1-\alpha \omega_1(z\iv)} + \frac{\theta_j}{\alpha}\right)_{j=1}^{p}.\]
\end{prop}

We now have the tools to give the proof for Theorem \ref{mult-real}
\begin{proof}\textit{(of Theorem \ref{mult-real})}

We first prove the model in the case where only $A_N$ as spikes, that is, for $X'_N =  A_N^{1/2}U_NB'_NU^*_NA_N^{1/2}$. Proposition \ref{4.3} guarantees the existence of the almost sure event on which there exists a sequence $\{\delta_N\}_N \in (0,\infty)$ converging to zero such that 
\begin{enumerate}[\indent {$-$}]
\item $\sigma((A'_N)^{1/2}U_NB'_NU^*_N(A'_N)^{1/2}) \subseteq K_{C\varepsilon + \delta_N}$, and 
\item the sequence $\{F_N\}_{N > p}$ converges to 
\[F(z) = \diag\left( \left( 1-\frac{\theta_j}{\alpha}\right)\frac{1}{1-\alpha \omega_1(z\iv)} + \frac{\theta_j}{\alpha}\right)_{j=1}^{p}\]
uniformly on compact sets of $\overline{\C} \bs K_{C\varepsilon}.$
\end{enumerate}

We want to use an application of Lemma \ref{4.5} with $\gamma = \R$, the sequence $\{F_N\}_{N >p}$, and the uniform on compacts limit $F$. We verify all the conditions of the Lemma \ref{4.5}. Conditions (1) and (3) follow directly from Proposition \ref{4.3}. To show condition (2), that is, show $F_N(z)$ is invertible for $z \notin \R$, we notice in equation (\ref{RMF}), that if $F_N(z)$ is not invertible, then $z$ is an eigenvalue of a self-adjoint matrix $X'_N$, and hence it is a real number. 

Lastly, we have $F(\infty) = I_p$, and since 
\[(F'(z))_{jj} = \frac{\omega'_1(1/z)(\theta_j - \alpha)}{z^2(1-\alpha \omega_1(1/z))^2}\]
and the zeros of $\omega'_1$ are simple (by the Julia-Carath\'eodory theorem), we have that the zeros $F$ are simple. Thus, Lemma \ref{4.5} applies to $F_N$ and $F$. 

For almost every $\delta >0$, we have that the boundary points of $K_{C\varepsilon + \delta}$ are not zeros of $\det(F)$. When this condition is satisfied, Lemma \ref{4.5} gives us exactly the results of Theorem 4.2, in the case where $B_N$ has no spikes. We saw above that the nonzero eigenvalues of $X'_N$ in $\C \bs K_{C \varepsilon + \delta}$ are exactly the zeros of $\det(F_N)$, and the set of points $z$ such that $F(z)$ is not invertible is precisely $\cup_{i=1}^{p} v_1\iv(\{1/\theta_i\})$. The case where $B'_N$ has spikes is completely analogous to the reasoning found in the proof of Theorem \ref{main}. 
\end{proof}

Finally, we present the multiplicative model $X_N = A_NU_NB_NU_N^*$ where $A_N$ and $B_N$ are unitary. For the following theorem, we use the notation that for $\rho \in \T$ and $\varepsilon >0$, the interval $(\rho - \varepsilon, \rho +\varepsilon)$ consists of elements in $\T$ whose argument differs from $\text{Arg}(\rho)$ by less than $\varepsilon$. We will use a result similar to Lemma \ref{convMultPos}.

\begin{lemma}\label{convMultT}
Let $A_N$ and $B_N$ are $N \times N$ independent Haar unitary matrices such that the laws of $A_N$ and $B_N$ are unitarily invariant. Let $\mu$ and $\nu$ be compactly supported measures on $\T$ such that almost surely $\mu_{A_N} \lra \mu$ and almost surely $\mu_{B_N} \lra \nu$. Let $\varepsilon > 0$ be given. Suppose there exists an $N_0 \in \N$ large enough such that,  for $N \geq N_0$, we have both
\[\sigma(A_N) \subseteq \text{\emph{Supp}}(\mu)_{\varepsilon}\]
\[ \sigma(B_N) \subseteq \text{\emph{Supp}}(\nu)_{\varepsilon}.\]
Then for $\eta>0$ there exists an $N_1$ such that for all $N \geq N_1$ and a constant $C >0$ such that $\sigma(A_NB_N) \subseteq  \text{\emph{Supp}}(\mu \boxtimes \nu)_{\eta +C\varepsilon}$.
\end{lemma}
\begin{proof} The proof follows in nearly the exact same way as the proof of Lemma \ref{convMultPos}. Suppose $N > N_0$. Due to our assumption of unitary invariance, we may assume $A_N$ and $B_N$ are diagonal matrices. We write
\[A_N = \diag(\theta_1, \theta_2, \dots, \theta_{k(N)}, \alpha_1, \dots, \alpha_{N-k(N)})\]
\[B_N = \diag(\tau_1, \tau_2, \dots, \tau_{\ell(N)}, \beta_1, \dots, \beta_{N-\ell(N)})\]
where $\theta_1, \dots, \theta_{k(N)}$ are the $k(N)$ elements in $(\lambda_{i}(A_N))_{i=1}^{N}$ that are outside the $\supp(\mu)_\varepsilon$, and $\tau_1, \dots, \tau_{\ell(N)}$ are the $\ell(N)$ elements in $(\lambda_{i}(B_N))_{i=1}^{N}$ that are outside the $\supp(\nu)_\varepsilon$. Define $a_i$ as a point in $\supp(\mu)$ such that $|\theta_i - a_i| = \dist(\theta_i, \supp(\mu))$, and similarly define $b_i$ as a point in $\supp(\nu)$ such that $|\tau_i - b_i| = \dist(\tau_i, \supp(\nu))$. Define
\[A'_N = \diag(a_1, \dots, a_{k(N)}, \alpha_1, \dots, \alpha_{N-k(N)})\]
\[B'_N = \diag(b_1, \dots, b_{k(N)}, \beta_1, \dots, \beta_{N-\ell(N)}).\]

Notice we have
\begin{align*}
\\||A_NB_N - A'_NB'_N||   &= ||A_NB_N - A'_NB_N + A'_NB_N - A'_NB'_N||\\
& \leq ||A_NB_N - A'_NB_N|| + ||A'_NB_N - A'_NB'_N||\\
& \leq ||B_N||||A_N - A'_N|| + ||A'_N||||B_N - B'_N|| \\
& \leq c_1 \varepsilon + c_2 \varepsilon \\
& = C\varepsilon 
\end{align*}
where $C, c_1, c_2$ are positive constants. Since we have that the $\sigma(A'_N) \in \supp(\mu)$ and $\sigma(B'_N) \in \supp(\nu)$ for all $N \geq N_0$, we know by Theorem \ref{cm} that for any  $\eta >0$ there exists an $N_1$ such that for all $N \geq N_1$ we have $\sigma(A'_NB'_N) \subseteq \supp(\mu \boxtimes \nu)_\eta$.

From \cite{tao2012topics}, we see that for any $i = 1, \dots, N$ we have 
\[|\lambda_i(A_NB_N) - \lambda_i(A'_NB'_N)| \leq ||A_NB_N - A'_NB'_N||.\]
Thus, it follows that for $N \geq \max\{N_1,N_0\}$ we have 
\[\sigma(A_NB_N) \subseteq  \text{\emph{Supp}}(\mu \boxtimes \nu)_{\eta + C\varepsilon}.\]
\end{proof}

\begin{theorem}
Suppose we have the following:
\begin{enumerate}[$1$.]
\item Two compactly supported Borel probability measures $\mu$ and $\nu$ on $\mathbb{T}$ with nonzero first moments such that $\text{\emph{Supp}}(\mu \boxtimes \nu) \neq \mathbb{T}$. 

\item A sequence of fixed complex numbers $\{\theta_i\}_{i=1}^{\infty} \subseteq \T$ such that:
\begin{enumerate}[\indent {$-$}]
\item $\text{\emph{Arg}} (\theta_i) \in [0, 2\pi)$for all $i = 1, 2, \dots$;
\item $\theta_i$ does not belong to $\text{\emph{Supp}}(\mu)$ for all $i = 1, 2, \dots$;
\item  \text{\emph{dist}}$(\theta_i, \text{\emph{Supp}}(\mu)) \lra 0$ as $i \lra \infty$. 
\end{enumerate}

\item A sequence $\{A_N\}_{N\in\N}$ of random Haar unitary matrices of size $N\times N$ such that:
\begin{enumerate}[\indent {$-$}]
\item $\mu_{A_N}$ converges weakly to $\mu$ as $N \lra \infty$;
\item a sequence $(\varphi(N))_{N \in \N}$ satisfying the conditions in  Proposition \ref{Qseq};
\item for $\theta \in \{\theta_1, \dots, \theta_{\varphi(N)}\}$, the sequence $(\lambda_n(A_N))_{n=1}^{N}$ satisfies 
\[\text{\emph{card}}(\{n : \lambda_n(A_N) = \theta\}) = \text{\emph{card}}(\{i : \theta_i = \theta\});\]
\item the eigenvalues of $A_N$ which are not equal to some $\theta_i$ converge uniformly to $\text{\emph{Supp}}(\mu)$ as $N \lra \infty$, that is 
\[\lim_{N \lra \infty} \max_{\lambda_n(A_N) \notin \{\theta_1,\dots, \theta_{\varphi(N)}\}} \text{\emph{dist}}(\lambda_n(A_N),\text{\emph{Supp}}(\mu)) = 0\]
\end{enumerate}

\item A sequence $\{U_N\}_{N \in \N}$ of unitary random matrices such that the distribution of $U_N$ is the normalized Haar measure on the unitary group U$(N)$.

\item A sequence of fixed real numbers $\{\tau_j\}_{j=1}^{\infty} \subseteq \T$ such that:
\begin{enumerate}[\indent {$-$}]
\item $\text{\emph{Arg}} (\tau_j) \in [0, 2\pi)$for all $i = 1, 2, \dots$;
\item $\tau_j$ does not belong to $\text{\emph{Supp}}(\nu)$ for all $j = 1, 2, \dots$;
\item  \text{\emph{dist}}$(\tau_j, \text{\emph{Supp}}(\nu)) \lra 0$ as $j \lra \infty$. 
\end{enumerate}

\item A sequence $\{B_N\}_{N\in\N}$ of random Haar Unitary matrices of size $N\times N$ such that:
\begin{enumerate}[\indent {$-$}]
\item $\mu_{B_N}$ converges weakly to $\nu$ as $N \lra \infty$;
\item a sequence $(\psi(N))_{N \in \N}$ satisfying the conditions in  Proposition \ref{Qseq};
\item for $\tau \in \{\tau_1, \dots, \tau_{\psi(N)}\}$, the sequence $(\lambda_n(B_N))_{n=1}^{N}$ satisfies 
\[\text{\emph{card}}(\{n : \lambda_n(B_N) = \tau\}) = \text{\emph{card}}(\{j : \tau_j = \tau\});\]
\item the eigenvalues of $B_N$ which are not equal to some $\tau_j$ converge uniformly to $\text{\emph{Supp}}(\nu)$ as $\N \lra \infty$, that is 
\[\lim_{N \lra \infty} \max_{\lambda_n(B_N) \notin \{\theta_1,\dots, \theta_{\psi(N)}\}} \text{\emph{dist}}(\lambda_n(B_N),\text{\emph{Supp}}(\nu)) = 0\]
\end{enumerate}
\end{enumerate}

Set $K = \text{\emph{Supp}}(\mu \boxtimes \nu)$, $v_k (z) = \omega_k(1/z)$ for $k = 1,2$, and 
\[\displaystyle K' = K \cup \left[ \bigcup_{i=1}^{\infty} v_1^{-1} (\{1/\theta_i\})\right] \cup \left[ \bigcup_{j=1}^{\infty} v_2^{-1} (\{1/\tau_j\})\right]\]
where $\omega_1,\omega_2$ are the subordination functions corresponding to the free convolution $\mu \boxtimes \nu$. 
Then,
\begin{enumerate}[$1$.]
\item Given $\varepsilon >0$, almost surely, there exists an $N_0 \in \N$ such that for all $N >N_0$, we have \[\sigma(X_N) \subset K'_{\varepsilon}\] where $X_N =  A_NU_NB_NU_N^*.$
\item  Fix a number $\rho \in K'\bs K$. Let $\varepsilon > 0 $ such that $(\rho - 2\varepsilon , \rho+2\varepsilon)\cap K' = \{\rho\}$ and set $k = \text{\emph{card}}(\{i:v_1(\rho) = 1/\theta_i\})$ and $\ell = \text{\emph{card}}(\{j:v_2(\rho) = 1/\tau_j\})$, then almost surely, there exists an $N_0 \in \N$ such that for all $N>N_0$, we have \[\text{\emph{card}}(\{\sigma(X_N)\cap (\rho-\varepsilon,\rho +\varepsilon)\}) = k+\ell.\]
\end{enumerate}
\end{theorem}

We use results analogous to Proposition \ref{4.3} and Lemma \ref{convMultT}. We get an identical result as Proposition \ref{4.3} where $R(z) = (zI_N - A_NU_NB_NU_N^*)\iv$, and Lemma \ref{convMultT} is used, but we must consider $z \in \C \bs \T$. The reduction to a $p \times p$ matrix is performed in the same way, but we choose $\alpha, \beta$ from $\T$ such that $1/\alpha \in \supp(\mu)$ and $1/\beta \in \supp(\nu)$.

\begin{proof}
We apply Lemma \ref{4.5} to our model, with $\gamma = \T$, the sequence $\{F_N\}_N$ defined by 
\[F_N(z) := (I_p - \Theta)P_N\left(zI_N - (A'_N)^{1/2} U_N B'_N U_N^* (A'_N)^{1/2} \right)\iv P_N^* + \Theta, \hspace{.5cm} z \in \C \bs \T\ ,\]
and the limit $F$ defined by 
\[F(z) = \diag\left( \left( 1-\frac{\theta_j}{\alpha}\right)\frac{1}{1-\alpha \omega_1(z\iv)} + \frac{\theta_j}{\alpha}\right)_{j=1}^{p}.\]

We notice that $F_N(z)$ is invertible for all $z \in \C \bs \T$, since if $F_N(z)$ is not invertible that $z$ must come from the spectrum of $A_NU_NB_NU_N^*$, which is contained $\T$. We have that $F_N$ converges on compact sets of $\C \bs \T$ by our modified version of Lemma \ref{bbcf5.1-mult}. The function $F$ is clearly diagonal, and once again, by the Julia-Carath\'eodory Theorem, this time applied to the disk, we have that the entries of $F$ have simple zeros. Hence we can apply Lemma \ref{4.5}, and remainder of the argument follows identically. 
\end{proof}


\section{Example and Further Work}
	
\begin{example}
We provide a numerical simulation where the number of spikes is increasing. Let the spikes come from the sets
\[S_1 = \left\{2+\frac{10}{n}\right\}_{n=1}^{10} \text{ and } S_2=\left\{2+\frac{10}{n}\right\}_{n=1}^{100}.\]
Consider our model given in Example \ref{example1}, except $N=2000$.  We display the histograms of the eigenvalues of one sample below in Figure 3 below. In part (A) we have that the spikes are the elements in $S_1$, and in part (B) the spikes are the elements from $S_2$. 

\begin{figure}[ht]
    \centering
    \begin{subfigure}[ht]{0.45\textwidth}
        \includegraphics[width=\textwidth]{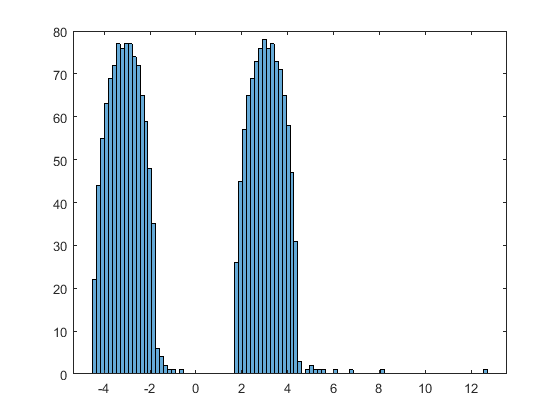}
        \caption{Spikes are $S_1$}
    \end{subfigure}
    ~ 
    \begin{subfigure}[ht]{0.45\textwidth}
        \includegraphics[width=\textwidth]{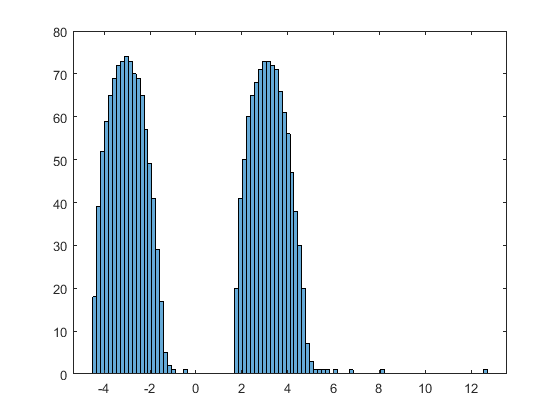}
        \caption{Spikes are $S_2$}
    \end{subfigure}
    ~ 
    \caption{Increasing number of spikes that approximate 2 from the right.}
\end{figure}
\end{example}

Further questions to explore in this area:
\begin{enumerate}
\item Do similar results hold for the case when the spikes of $A_N$ and $B_N$ accumulate a positive distance from $\supp(\mu)$ and $\supp(\nu)$?
\item Do similar results hold for the case when the spikes simply do not converge?
\end{enumerate}
For these questions, the reduction to a $p \times p$ matrix technique used in here would not be applicable, and thus they do not appear to be simple extensions of this result.

\newpage

\nocite{nica2006lectures}
\nocite{mingo2017free}
\nocite{haagerup2005new}
\nocite{voiculescu1992free}
\nocite{tao2012topics}
\nocite{capitaine2013additive}

\bibliographystyle{acm}
\bibliography{biblio}

\begin{thebibliography}{10}

\bibitem{anderson2010introduction}
{\sc Anderson, G.~W., Guionnet, A., and Zeitouni, O.}
\newblock {\em An introduction to random matrices}, vol.~118.
\newblock Cambridge university press, 2010.

\bibitem{belinschi2017outliers}
{\sc Belinschi, S.~T., Bercovici, H., Capitaine, M., Fevrier, M., et~al.}
\newblock Outliers in the spectrum of large deformed unitarily invariant
  models.
\newblock {\em The Annals of Probability 45}, 6A (2017), 3571--3625.

\bibitem{capitaine2013additive}
{\sc Capitaine, M.}
\newblock Additive/multiplicative free subordination property and limiting
  eigenvectors of spiked additive deformations of wigner matrices and spiked
  sample covariance matrices.
\newblock {\em Journal of Theoretical Probability 26}, 3 (2013), 595--648.

\bibitem{collins2014strong}
{\sc Collins, B., and Male, C.}
\newblock The strong asymptotic freeness of haar and deterministic matrices.
\newblock {\em Ann. Sci. {\'E}c. Norm. Sup{\'e}r.(4) 47}, 1 (2014), 147--163.

\bibitem{gamelin2003complex}
{\sc Gamelin, T.}
\newblock {\em Complex analysis}.
\newblock Springer Science \& Business Media, 2003.

\bibitem{garnett2007bounded}
{\sc Garnett, J.}
\newblock {\em Bounded analytic functions}, vol.~236.
\newblock Springer Science \& Business Media, 2007.

\bibitem{haagerup2005new}
{\sc Haagerup, U., and Thorbj{\o}rnsen, S.}
\newblock A new application of random matrices: \emph{$Ext(C^*_{red}(F_2))$} is
  not a group.
\newblock {\em Annals of Mathematics\/} (2005), 711--775.

\bibitem{hogben2006handbook}
{\sc Hogben, L.}
\newblock {\em Handbook of linear algebra}.
\newblock Chapman and Hall/CRC, 2006.

\bibitem{male2012norm}
{\sc Male, C.}
\newblock The norm of polynomials in large random and deterministic matrices.
\newblock {\em Probability Theory and Related Fields 154}, 3-4 (2012),
  477--532.

\bibitem{mingo2017free}
{\sc Mingo, J.~A., and Speicher, R.}
\newblock {\em Free probability and random matrices}, vol.~4.
\newblock Springer, 2017.

\bibitem{nica2006lectures}
{\sc Nica, A., and Speicher, R.}
\newblock {\em Lectures on the combinatorics of free probability}, vol.~13.
\newblock Cambridge University Press, 2006.

\bibitem{speicher1993free}
{\sc Speicher, R.}
\newblock Free convolution and the random sum of matrices.
\newblock {\em Publications of the Research Institute for Mathematical Sciences
  29}, 5 (1993), 731--744.

\bibitem{tao2012topics}
{\sc Tao, T.}
\newblock {\em Topics in random matrix theory}, vol.~132.
\newblock American Mathematical Soc., 2012.

\bibitem{voiculescu1986addition}
{\sc Voiculescu, D.}
\newblock Addition of certain non-commuting random variables.
\newblock {\em Journal of functional analysis 66}, 3 (1986), 323--346.

\bibitem{voiculescu1987multiplication}
{\sc Voiculescu, D.}
\newblock Multiplication of certain non-commuting random variables.
\newblock {\em Journal of Operator Theory\/} (1987), 223--235.

\bibitem{voiculescu1991limit}
{\sc Voiculescu, D.}
\newblock Limit laws for random matrices and free products.
\newblock {\em Inventiones mathematicae 104}, 1 (1991), 201--220.

\bibitem{voiculescu1992free}
{\sc Voiculescu, D.~V., Dykema, K.~J., and Nica, A.}
\newblock {\em Free random variables}.
\newblock No.~1. American Mathematical Soc., 1992.

\end{thebibliography}

\end{document}